\documentclass[11pt]{amsart}
 
\usepackage{bbm, amsmath, amsfonts, amscd, latexsym, amsthm, amssymb, graphicx,hyperref,mathrsfs, comment}
\usepackage{abstract}
\usepackage{mathtools}
\usepackage{caption}
\usepackage[dvipsnames]{xcolor}

\usepackage{tikz-cd}
\tikzset{
  symbol/.style={
    draw=none,
    every to/.append style={
      edge node={node [sloped, allow upside down, auto=false]{$#1$}}}
  }
}

\usepackage[usenames,dvipsnames]{pstricks}

\newtheorem{theor}{\hspace{1cm}{\sc Theorem}}[section]

\newtheorem{sledst}[theor]{\hspace{1cm}{\sc Corollary}}
\newtheorem{lemma}[theor]{\hspace{1cm}{\sc Lemma}}

\newtheorem*{utver*}{\hspace{1cm}{\sc Proposition}}
\theoremstyle{definition}

\newtheorem{defin}[theor]{\hspace{1cm}{\sc Definition}}
\newtheorem{exa}[theor]{\hspace{1cm}{\sc Example}}
\newtheorem{observ}[theor]{\hspace{1cm}{\sc Observation}}
\newtheorem{rem}[theor]{\hspace{1cm}{\sc Remark}}

\newtheorem*{conven}{\hspace{1cm}{\sc Convention}}

\newcommand{\Vol}{\mathop{\rm Vol}\nolimits}

\newcommand{\id}{\mathop{\rm id}\nolimits}
\newcommand{\im}{\mathop{\rm im}\nolimits}

\newcommand{\tN}{\widetilde N}

\newcommand{\coker}{\mathop{\rm coker}\nolimits}

\renewcommand{\emph}[1]{{\it {\color{NavyBlue} #1}}}
\newcommand{\h}[1]{H_1(#1)}

\def\sym{\mathfrak{S}}
\def\Hom{\text{Hom}}
\def\tT{\widetilde{T}}
\def\tG{\widetilde{G}}

\def\tA{\tilde{A}}
\def\tU{\widetilde{U}}
\def\tV{\widetilde{V}}
\def\R{\mathbb R}

\def\Z{\mathbb Z}

\def\C{\mathbb C}
\def\CC{({\mathbb C}^\star)}
\def\cC{{\mathbb C}^\star}

\def\CP{\mathbb C\mathbb P}

\begin{document}
\begin{center}{\Large \sc Sparse polynomial equations and other enumerative problems whose Galois groups are wreath products}

\vspace{3ex}

{\sc A. Esterov\footnote{Supported by the Russian Science Foundation grant, project 16-11-10316.
}, L. Lang}
\end{center}

\begin{abstract} 
We introduce a new technique to prove connectivity of subsets of covering spaces (so called inductive connectivity), and apply it to Galois theory of problems of enumerative geometry.

As a model example, consider the problem of permuting the roots of a complex polynomial $f(x) = c_0 + c_1 x^{d_1} + \ldots + c_k x^{d_k}$ by varying its coefficients. If the GCD of the exponents is $d$, then the polynomial admits the change of variable $y=x^d$, and its roots split into necklaces of length $d$. At best we can expect to permute these necklaces, i.e. the Galois group of $f$ equals the wreath product of the symmetric group over $d_k/d$ elements and $\Z/d\Z$.

The aim of this paper is to prove this equality and study its multidimensional generalization: we show that the Galois group of a general system of polynomial equations equals the expected wreath product for a large class of systems, but in general this expected equality fails, making the problem of describing such Galois groups unexpectedly rich.

\end{abstract}

\section{Introduction and main results}

\subsection{Sparse polynomial equations.}
Consider the set $A$ of integer numbers $$0=a_0<a_1<\ldots<a_k=a$$
and the space $\C^A$ of all complex polynomials of the form
\begin{equation}\label{eq:univariate}
    f(x) = c_0 + c_1 x^{a_1} + \ldots + c_k x^{a_k}.
\end{equation}{}

As $f$ travels along a loop in $\C^A$, avoiding the discriminant $D=\{$polynomials wih less than $a$ roots$\}$, its roots undergo a permutation. The group of all such permutations $G_A$ is called the \emph{monodromy group}, or the \emph{Galois group} of the polynomial. 

\begin{rem} The latter name comes from the fact that $G_A$ equals the Galois group of the corresponding extension of the field $\C(c_0,\ldots,c_k)$. In this way the definition of the group $G_A$ extends from polynomials over complex numbers to an arbitrary field. Over fields of positive characteristic, this group has been thoroughly studied especially for trinomials ($k=2$), but the complete answer is not known so far: \cite{tri1}, \cite{cohen80}, \cite{tri2}, \cite{tri3}. A well known related project, originating from \cite{abh1} and culminating in \cite{abh2}, is Abhyankar's identification of many interesting groups as Galois groups of trinomials whose coefficients are powers of the same parameter. In what follows, we restrict ourselves to the complex setting. \end{rem}

The following fact is widely known, see e.g. \cite{cohen80} or \cite{E18} 
for the algebraic and the topological proof respectively.
\begin{observ}\label{thdim1red}
If $a_1,\ldots,a_k$ are mutually prime, then $G_A=\sym_a$.
\end{observ}
The assumption here cannot be relaxed: if $d:={\rm GCD}(a_1,\ldots,a_k)$ is greater than 1, then the group $G_A$ is strictly smaller than $\sym_a$, because in this case the equation $f(x)=0$ has the form $\tilde f(x^d)=0$. In particular, every root $y$ of $\tilde f$ gives rise to the necklace of $d$ roots of $f$ of the form $\sqrt[d]{y}$, and at best we can expect to permute these $a/d$ (oriented) necklaces. 
The group of permutations of a disjoint union of necklaces is the simplest example of a wreath product.

\begin{defin}\label{defwrpr} Let $S$ be a finite set, $H$ and $G$ be two groups and $\varphi : G \rightarrow \sym(S)$ an action of $G$ on $S$. The \emph{wreath product} $H \wr_\varphi G$ is 
the semidirect product of $H^S$ and $G$ with respect to the action of $G$ on $H^d$ given by $g\cdot f= f\circ \varphi(g)$ for any element $f:S\rightarrow H$ of $H^S$.

\end{defin}
\begin{rem}
We will  mostly consider the case when $\varphi$ is injective, that is $G$ is a subgroup of the group of permutation $\sym(S)$ on $S$. In such case, we will simply denote the wreath product by $H\wr G$. 
This group can be seen as the group of all permutations $\sigma$ of the set $H\times S$, satisfying the following properties:

1) $\sigma$ can be included into the commutative diagram

$$    \begin{tikzcd}
H\times S \arrow{r}{\sigma} \arrow{d} & H\times S \arrow{d} \\
S  \arrow{r} & S 
\end{tikzcd}$$

where the vertical arrows are the standard projection, and the bottom arrow belongs to $G$.

2) The restriction $\sigma:H\times\{i\}\to H\times\{j\}$ for every $i$ is the multiplication by an element of $H$.

In particular, if $H$ is cyclic of order $d$, then $H\wr\sym_a$ is the group of permutations of $a$ oriented necklaces of length $d$.
\end{rem}

\begin{theor}\label{thdim1nonred}
For $d:={\rm GCD}(a_1,\ldots,a_k)$, we have $G_A=(\Z/d\Z)\wr\sym_{a/d}$, i.e. the monodromy group includes all permutations of the roots of the equation $\eqref{eq:univariate}$ that preserve the necklace structure.
\end{theor}

\begin{rem}
It would be interesting to give this theorem an algebraic proof and to understand to what extent it survives the positive characteristic.
\end{rem}

We shall deduce this theorem as a special case of a certain new general fact about the monodromy of enumerative problems: it is a consequence of Theorem \ref{mainth} and Lemma \ref{ldim1nonred} below. After that, we discuss the generalization of this fact to systems of several equations.

\subsection{Structure of the paper.} The paper contains three interrelated results.

(1) The aforementioned Theorem \ref{mainth} on the monodromy of enumerative problems is presented in the Introduction and proved in Section 2.

(2) Its proof is based on a criterion for connectivity of the inverse image of a connected set under a covering map. This criterion is presented in the Intorduction  (Theorem \ref{topolth0}) and proved in Section 2.

(3) Finally, we generalize Theorem \ref{thdim1nonred} to systems of several sparse polynomial equations. A light two-dimensional version (Theorem \ref{thnonred0}) is presented in the Introduction, the full version is presented and deduced from Theorem \ref{mainth} in Section 3.

In contrast to the univariate case, there exist systems of several equations whose Galois group is not equal to the expected wreath product, see Example \ref{ex:smaller}.

\subsection{Enumerative problems.}
The general enumerative problem consists of a pair of smooth algebraic varieties $T$ and $C$ with $C$ connected, together with an algebraic set $U\subset T\times C$ such that the projection $c:U\to C$ is generically finite. The sets $T$ and $C$ are respectively referred to as the \emph{ambient space} and the \emph{space of conditions} of the enumerative problem. 

The set $U$ is called the \emph{solution space}, and the points of the fiber $c^{-1}(f)$ (as well as their images under the projection $t:U\to T$) are called the \emph{solutions} of the enumerative problem for the condition $f\in C$.

For instance, the projection of a complex plane curve $U$ to the horizontal axis $C$ along the vertical axis $T$ is the simplest example of an enumerative problem:
\begin{center}
\includegraphics[width=4cm]{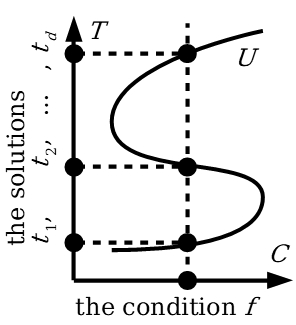}
\end{center}

Here are some other important enumerative problems.
\begin{exa}\label{exaenum} 1) The setting of the first subsection can be regarded as an enumerative problem with $T=\C,\; C=\C^A$, and $U$ defined as the set of all pairs $(x,f)$ satisfying the equation $f(x)=0$.

2) (Pl\"ucker). Let $T$ be the projective plane, $C$ be the space of degree $d$ homogeneous polynomials on $T$, and $U$ be the set of pairs $(x,f)$ such that the plane algebraic curve $f=0$ has a flex at $x$. 

3) (Schubert). Let $T$ be the space of lines in $\CP^n$, $C$ be the space of collections $(V_{1},\ldots,V_{k})$, where $V_i$ is a linear   subspace of $\CP^n$ of codimension $a_i$, and $U$ be the set of collections $(l,V_{1},\ldots,V_{k})$ such that the line $l$ intersects every $V_i$. This is an enumerative problem for $\sum_i (a_i-1)=2n-2$.

\end{exa}

The first step in the study of an enumerative problem is to compute the number of solutions for a generic condition $f\in C$. This number is called the \emph{degree} of the enumerative problem. For the problems in the preceding example, this question was addressed by the aforementioned authors. 

The next natural step is to study the Galois group of the enumerative problem. More specifically, the projection $c:U\to C$ is a covering over a sufficiently small 
Zariski open set $C_\circ\subset C$, i.e. $c^{-1}(C_\circ)\to C_\circ$ is a covering. 
Denote its total space $c^{-1}(C_\circ)$ by $U_\circ\subset U$.
\begin{defin} 
The 
monodromy group of the covering $U_\circ\to C_\circ$ is called the \emph{monodromy group} or the \emph{Galois group} of the enumerative problem $U\to C$. 
\end{defin}

\begin{conven}
The Galois group of the enumerative problem $U\to C$ does not depend on the choice of $C_\circ$. Passing from the enumerative problem $U\to C$ to $U_\circ\to C_\circ$, we can and will always assume throughout the paper that every enumerative problem $U\to C$ is a covering. 

Under this assumption, the variety $U$ is smooth. Therefore, the irreducible components of $U$ are exactly its connected components, so we shall refer to them simply as \emph{components}.
\end{conven}

For the last two problems in Example \ref{exaenum}, the study of the Galois group was initiated in the foundational papers \cite{Harris} and \cite{Vakil} respectively.

\subsection{Wreath enumerative problems.}
Consider an enumerative problem  $U\subset T\times C$ and a covering $\phi:\tilde T\to T$. 
\begin{defin} The preimage $\tU$ of the solution space $U$ under the covering map $(\phi,\id):\tT\times C\to T\times C$ defines the enumerative problem $\tU\to C$ that we shall call the \emph{$\phi$-wreath} over $U$.
\end{defin}
\begin{exa}
Let $\phi:\CC\to\CC$ be given by $\phi(x)=x^d$, then the $\phi$-wreath over the enumerative problem of Example \ref{exaenum}.1
$$T=\CC,\; C=\C^A,\; U=\{(x,f)\,|\, f(x)=0\}$$
is the enumerative problem 
$$\tilde T=\CC,\; C=\C^{\tilde A},\; \tilde U=\{(\tilde x,\tilde f)\,|\, \tilde f(\tilde x)=0\},$$
where $\tilde A=d\cdot A=\{0,da_1,\ldots,da_k\}$.
\end{exa}
Our primary goal in this paper is to study the monodromy group of wreath enumerative problems, motivated by the following observation.
\begin{observ}
It is a widely known empirical phenomenon that Galois groups of interesting enumerative problems (including the ones mentioned in Example \ref{exaenum}, see for instance \cite{schsot}, \cite{Tyom14}, \cite{L19} and \cite{E18}) are either imprimitive or ``trivial'' (i.e. symmetric or alternating).  
On the other hand, if the Galois group of an enumerative problem $U$ is imprimitive, then the covering $U\to C$ decomposes into a non-trivial composition of covering maps, which is close to being a wreath enumerative problem over a simpler one.

Thus, it seems that the study of non-trivial Galois groups of enumerative problems to a large extent reduces to the case of wreath problems. The subsequent theorems estimate and compute Galois groups of $\phi$-wreath problems for Galois coverings $\phi$.
\end{observ}
\begin{observ} \label{mainth0}
Consider an enumerative problem $U\subset T\times C$ of degree $d$ with Galois group $G$, and a Galois covering $\phi:\tT\to T$ with group of deck transformations $D$.
Then the Galois group $\tG$ of the $\phi$-wreath enumerative problem $\tU$ is contained in $W:=D\wr G$.
\end{observ}
This observation directly follows from the definitions of the wreath problem and the wreath group, and is widely known in the algebraic Galois theory (the Galois group of an iterated finite separable field extension is a subgroup of the corresponding wreath product, see e.g. \cite{wreathgalois}). A less straightforward task is to find criteria for the equality $\tG=W$. We present a criterion based on the following notions.

\subsection{Inductive covers and solution lattices.}

In what follows, all singular homology groups are assumed to be with integer coefficients, so we omit the coefficient ring in the notation.

\begin{defin}
A covering of path-connected topological spaces $\pi:X\to Y$ is said to be \emph{inductive}, if the natural embedding $\pi_1'(X)\to\pi_1'(Y)$ is an isomorphism (the prime stands for the commutator subgroup).
\end{defin}

\begin{rem}\label{reminduct}

1) For an inductive covering $\pi:X\to Y$, the map $\pi_*:\h{X}\to\h{Y}$ is injective. Indeed, according to the diagram \begin{equation}\label{eq:comdiag}
    \begin{tikzcd}
\pi_1(X) \arrow[r,symbol=\longrightarrow] \arrow[d,symbol=\hookrightarrow] & H_1(X) \arrow{d}{\pi_*} \\
\pi_1(Y)  \arrow[r,symbol=\longrightarrow] & H_1(Y) 
\end{tikzcd},
\end{equation} 
a non-zero element in the kernel of $\pi_*$ would lift to an element in $\pi_1(X)\setminus \pi_1'(X)$ mapping to the kernel of $\pi_1(Y)\rightarrow H_1(Y)$, that is $\pi_1'(Y)$.

2) Note that the preceding implication is not an equivalence: if $X\to Y$ is the universal covering of a bouquet of circles, then the map $0=H_1(X)\to H_1(Y)$ is injective, while $0=\pi_1'(X)\to\pi_1'(Y)$ is not an isomorphism.

3) Every inductive covering is 
Galois with a commutative deck transformation group. Indeed, every subgroup $H$ of a group $G$ containing the commutator is normal, and the quotient $G/H$ is commutative. Applying this fact to the image of $\pi_1(X)$ in $\pi_1(Y)$, the sought statements follow from Propositions 1.39 and 1.40 in \cite{Hatcher}.

4) Note that the preceding implication is not an equivalence: if $X\to Y$ is the non-trivial two-fold covering of the bouquet of two circles, whose restriction to one of the circles is trivial, then it is Galois with the group of deck transformations $\Z/2\Z$, but not inductive (as the map $H_1(X)\to H_1(Y)$ is not injective).
\end{rem}

\begin{defin}
An inductive covering $\pi:X\to Y$ is said to be \emph{strongly inductive}, if the embedding $H_1^T(X)\to H_1^T(Y)$ (see Remark \ref{reminduct}.1) is an isomorphism ($T$ stands for the torsion subgroup).
\end{defin}
The following sources of strongly inductive coverings will be especially important for us.
\begin{exa}\label{rem:algebraicgroups}\label{exastrongind}
1) The map $\CC\to\CC,\, x\mapsto x^d$, is strongly inductive.
 
2) If $X\to Y$ is strongly inductive, then so is $X\times Z\to Y\times Z$.

3) The composition of (strongly) inductive coverings is (strongly) inductive.

4) 
Conversely, if a (strongly) inductive covering $\pi:X\to Y$ decomposes into covering maps $X\rightarrow \tilde X \rightarrow Y$, then both of them are (strongly) inductive.
Indeed, we have the factorization $\pi_1'(X)\rightarrow\pi_1'(\tilde{X})\rightarrow\pi_1'(Y)$. Since the composition of the two arrows is an isomorphism by assumption, and each arrow is injective, it is an isomorphism as well. Similarly, each arrow in $H_1^T(X)\rightarrow H_1^T(\tilde{X})\rightarrow H_1^T(Y)$ is an isomorphism. 

5) Every covering of a space with commutative fundamental group (in particular, every connected algebraic group covering over $\C$) is inductive. 
A particular instance is given by surjective morphisms $\pi:X\to Y$ between complex tori $X$ and $Y$ of the same dimension. Since $\h{X}$ and $\h{Y}$ have no torsion, the covering $\pi$ is strongly inductive. 
\end{exa}

Given an enumerative problem defined by a degree $d$ covering $c:U\to C$, choose a base point $f$ in $C$, and denote the set $c^{-1}(f)$ of its $d$ solutions by $S$. Consider an $f$-pointed loop $\gamma$ in $C$ whose induced permutation on $S$ is the identity. Then, along this loop, every solution $s\in S$ travels a loop in $T$, and the homology class of this loop will be denoted by $\gamma_s\in H_1(T)$. Ordering the elements $\gamma_s$ accordingly to a given ranking  
$\delta:S\to\{1,2,\ldots,d\}$, we obtain a vector $\gamma_\delta$ in $H_1(T)^d$.
\begin{defin}
For a given ranking $\delta$, the set of all vectors of the form $\gamma_\delta$ is called the \emph{solution lattice} $H_\delta\subset H_1(T)^d$.
\end{defin}
\begin{rem}\label{remsollat}
1) The solution lattice $H_\delta$ is a subgroup in $H_1(T)^d$ since it is the image of the homomorphism $\gamma \mapsto \gamma_\delta$ defined on the kernel of the monodromy map $\pi_1(C)\rightarrow \sym(S)$.

2) The lattices $H_\delta$ for various $\delta$ differ by the permutations of the multipliers in $H_1(T)^d$.
\end{rem}

\begin{theor}\label{mainth}
In the setting of Observation \ref{mainth0}, assume additionally that $\phi$ is strongly inductive. Then, the inclusion $\tG\subset W$ is an equality if and only if, for some ranking $\delta$ (or, equivalently, for every ranking), we have 
\begin{equation}\label{eq:solutionlattice}
    \h{T}^d=H_\delta+\phi_* \h{\tT}^d.
\end{equation}
\end{theor}
The proof is given in Subsection \ref{proofofmainth}. The importance of this theorem is that the study of (non-commutative) fundamental groups is replaced with the study of (commutative) homology. As an illustration of this advantage, we now prove Theorem \ref{thdim1nonred} by constructing the solution lattice of the underlying enumerative problem. 
\subsection{Application to sparse polynomial equations.}
Theorem \ref{thdim1nonred} is a corollary of the preceding theorem and the following lemma.
\begin{lemma}\label{ldim1nonred}
In the setting of Example \ref{exaenum}.1, if ${\rm GCD}(A)=1$, then the solution lattice equals $H_1\CC^a=\Z^a$.
\end{lemma}
\begin{proof}
I. By Observation \ref{thdim1red}, the Galois group of the enumerative problem is full symmetric, thus the solution lattice is symmetric under the natural action of $\sym_a$.

II. It follows from $(I)$ and Remark \ref{remsollat}.1 that, once the solution lattice contains vectors 
$$(\underbrace{1,\ldots,1}_{p},0,\ldots,0) \mbox{ and } (\underbrace{1,\ldots,1}_{q},0,\ldots,0)\in\Z^a,$$
it contains the vector $(\underbrace{1,\ldots,1}_{|p-q|},0,\ldots,0)$.

III. For every $a_j\in A$, the solution lattice contains the vector
\begin{equation}\label{eq:vector}
    (\underbrace{1,\ldots,1}_{a_j},0,\ldots,0)\in\Z^a.
\end{equation}
To prove this, consider the trinomial $g(x)=\varepsilon+x^{a_j}+x^a$. It has $a_j$ small roots (tending to 0 as $\varepsilon\to 0$) and $a-a_j$ other roots. As $\varepsilon$ runs a small loop around $0$, the small roots permute in a cycle around $0$, and the other roots travel small loops not linked with $0$: 
\begin{center}
\includegraphics[width=4cm]{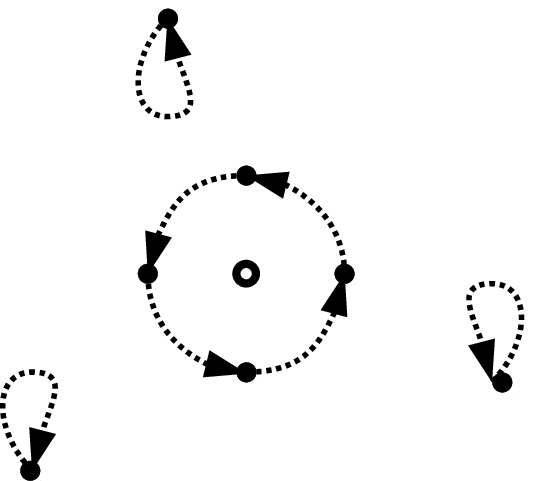}
\end{center}

Thus, if $\varepsilon$ runs $a_j$ times around $0$, then the roots of $g$ permute trivially, generating the element \eqref{eq:vector} in the solution lattice.

IV. Using (II), we can run the Euclidean algorithm on the vectors \eqref{eq:vector} for all $j$ and conclude that the solution lattice contains the vector $$(\underbrace{1,\ldots,1}_{{\rm GCD}(A)},0,\ldots,0).$$
Since we are given ${\rm GCD}(A)=1$, we have proved that $(1,0,\ldots,0)\in H_\delta$ and hence by $(I)$ the other vectors of the standard basis are contained in the solution lattice.
\end{proof}
\begin{rem}
The proof essentially consists of constructing loops in the space of conditions, generating the sought solution lattice. One could try instead to prove Theorem \ref{thdim1nonred} by directly constructing loops generating the sought monodromy group, and observe that this task is drastically more complicated (cf. our subsequent paper \cite{el2}).
\end{rem}

\subsection{Inductive connectivity}

The proof of Theorem \ref{mainth} is based on the following purely topological observation.
\begin{theor}\label{topolth0} Assume that $\pi:X\to Y$ is a strongly inductive covering of finite degree. The preimage of a path-connected subset $V\xhookrightarrow{j}Y$ is path-connected if and only if  $$j_*\h{V}+\pi_*(\h{X})=\h{Y}.$$
\end{theor}
The proof is given below in the form of Theorem \ref{corolmain}. It is based on decomposing the covering $\pi$ into a composition of cyclic covers and proving a stronger statement by induction on the length of this decomposition (hence the name ``inductive'').
\begin{rem}
1) Since irreducible components of a complex algebraic set are in one to one correspondence with the connected components of its smooth part, Theorem \ref{topolth0} may help to prove that the preimage of an irreducible algebraic variety $\,\mathcal{V}\subset Y$ under an algebraic cover $X\to Y$ is irreducible, when applied to the smooth part (or any sufficiently small Zariski open subset) of $\,\mathcal{V}$.

2) Since in what follows we apply this theorem to algebraic covers, we are fine with the assumption of finite degree. However, it would be interesting to understand to what extent one can omit assumptions on the cardinality of a fiber.

\end{rem}

\begin{exa}
The assumption of strong inductivity cannot be relaxed to inductivity in Theorem \ref{topolth0}.
For instance, consider the lens space $Y$ defined as the quotient of $X=(\C^2\setminus 0)$ by $(e^{2\pi i / p}, e^{2\pi i q / p})$. We have $\pi_1(X)=H_1(X)=0$ and $\pi_1(Y)=H_1(Y)=\Z/p\Z$.
For every line $V \subset \C^2$ through 0, its image $V\xhookrightarrow{j}Y$ generates the same subgroup $j_* H_1(V) \subset H_1(Y)$ for all possible choices of $V$, but the preimage of V may be reduced or not (and with different number of components) depending on $V$.
\end{exa}

\subsection{Systems of sparse equations}\label{ssparsesys}

We now discuss the generalization of Theorem \ref{thdim1nonred} to systems of several equations.

Identifying a point $a:=(a_1,\ldots,a_n)\in\Z^n$ with the monomial $x^a:=x_1^{a_1}\ldots x_n^{a_n}$, every finite set $P\subset\Z^n$ gives rise to the vector space $\C^P=\{\sum_{a\in P} c_ax^a\}$ of \emph{Laurent polynomials supported at $P$}. Every such polynomial defines a function on \emph{the complex torus} $T:=\CC^n$.

For a tuple of finite sets $A:=(A_1,\ldots,A_n),\, A_i\subset\Z^n$, a tuple $f:=(f_1,\ldots,f_n)$ from the space $\C^A:=\C^{A_1}\oplus\ldots\oplus\C^{A_n}$ can be regarded as a system of polynomial equations $f=0$ in the torus $T$. According to the Kouchnirenko--Bernstein theorem \cite[Theorem A]{bernst}, the number $d$ of solutions of this system equals the mixed volume of the convex hulls of $A_1,\ldots,A_n$ unless the tuple $f$ belongs to a certain proper Zariski closed subset $D\subset\C^A$, called the \emph{bifurcation set}.

We shall study the \emph{$A$-enumerative problem} with ambient space $T=\CC^n$, space of conditions $C=\C^A\setminus D$ and space of solutions $U=\{(x,f)\,|\, f(x)=0\}\subset T\times C$. 
In particular, we shall generalize Theorem \ref{thdim1nonred} to the Galois group $G_A\subset \sym_d$ of this enumerative problem.

\begin{theor}[{\cite[Theorem 1.5]{E18}}]\label{thred0} The Galois group $G_A$ equals the symmetric group $\sym_d$, if $A$ is reduced and irreducible in the following sense.
\end{theor}
\begin{defin}\label{def:redirred}
A tuple of finite sets $A:=(A_1,\ldots,A_n),\, A_i\subset\Z^n$, (and the corresponding system of equations with indeterminate coefficients) is said to be \emph{non-reduced} if all sets can be shifted to the same proper sublattice, and \emph{reducible} if $k$ of them can be shifted to a rank $k$ sublattice for some $k<n$.
\end{defin}
\begin{rem}
Theorem \ref{thred0} allows to classify systems of sparse equations solvable by radicals, see \cite{E18} for details. More precisely, it reduces this problem to the classification of lattice polytopes of small mixed volume, see e.g. \cite{abs} for recent advances in this direction. However, the complete computation of Galois groups is an open question both for non-reduced and for reducible systems, even in the simplest non-trivial case of a system of two trinomial equations of two variables.
\end{rem}
\begin{exa} The equation with indeterminate coefficients $c_8x^8+c_4x^4+c_0=0$ supported at $\{0,4,8\}$ is non-reduced.  The system $f(x)=g(x,y)=0$ is reducible.
\end{exa}
\begin{rem} 1) There is no loss of generality in assuming that $0\in A_i$ for $i=1,\ldots,n$. Indeed, we can divide every equation of a system by a certain monomial so that the resulting system satisfies the above assumption. Since dividing by a monomial does not affect the roots of the system in the complex torus, the resulting system has the same monodromy as the initial one. Therefore, we will always work under this assumption throughout this paper.

2) Under this assumption, we can interpret non-reduced systems as systems that can be simplified by a monomial change of variables, and reducible systems as systems that have a proper square subsystem of equations upon an appropriate monomial change of coordinates, as in the preceding example.

3) The most natural source of systems of sparse equations comes from the theory of Newton polytopes, when $A_1=\ldots=A_n$ is the set of lattice points of a lattice polytope $P$. Note that even in this case, starting from dimension 3, such systems may be non-reduced, i.e. the lattice points of the polytope $P$ may not generate the ambient lattice. Such polytopes are called non-reduced or non-spanning, see e.g. \cite{ht} for the classification up to volume 4.
\end{rem}
\subsection{Expected Galois groups of non-reduced systems.}\label{ssparsesysexp}
The simplification of a non-reduced system of equations with a monomial change of coordinates presents the corresponding enumerative problem as a wreath problem over a simpler one. Namely, let $\tN \simeq \Z^n$ be a lattice and  $\tilde A:=(\tilde A_1,\ldots,\tilde A_n)$, $\tilde A_i\subset \tN$, be a non-reduced irreducible tuple. Then, there exist a lattice $N\simeq \Z^n$, a reduced irreducible tuple $A=(A_1,\ldots,A_n)$, $A_i\subset N$, and a proper linear embedding $L:N\to\tN$ such that $\tilde A_i=L(A_i)$ for $i=1,\dots,n$. We call the tuple $A$ a \emph{reduction} of $\tilde A$. 
\begin{rem}\label{rem:reduction}
The reduction $A$ of $\tilde A$ is unique up to affine linear automorphism of $N$. 

\end{rem}
Given a reduction $A$ of a non-reduced tuple $\tA=L(A)$, denote by $T:=\Hom(N,\cC)$, by $\tT:=\Hom(\tN,\cC)$, and by $\phi_L: \tT \rightarrow T$  the map of tori induced by $L$. 
In coordinates, this is a monomial change of variables that takes every $f\in\C^A$ to $\tilde f(x):=f(\phi_L(x))\in\C^{\tA}$. The map assigning $\tilde f$ to $f$ is an isomorphism of vector spaces $\C^A\to\C^{\tilde A}$. Under this identification, the $\tA$-enumerative problem is the $\phi_L$-wreath over the $A$-enumerative problem. In particular, Theorem \ref{thred0} ensures that the Galois group $G_{\tA}$ is contained in the wreath product 
\begin{equation}\label{eq:wreathprod}
    \ker\phi_L\wr\sym_{d}.
\end{equation}
\begin{rem}\label{rem:iso} Note that the following three groups are canonically isomorphic to each other, and non-canonically isomorphic to $\coker L$:

-- $\ker \phi_L \subset T$;

-- $\coker L^*$, where $L^*$ is the lattice embedding dual to $L$;

-- the Pontryagin dual to the group $\coker L$.
\end{rem}

In contrast to the one-dimensional case, the embedding of the Galois group $G_A$ to \eqref{eq:wreathprod} can be proper (see an Example \ref{ex:smaller} below), but the equality still holds for a large class of systems of equations.

For a finite subset $B\in\Z^m$, denote the intersection of $B$ with the boundary of its convex hull by $\partial B$.
\begin{theor}\label{thnonred0}
Assume that $n=2$, $\tilde A_1=\tilde A_2=:B$, and (with no loss of generality) that $0\in\partial B$. Then the Galois group $G_{\tA}$ equals the expected wreath product \eqref{eq:wreathprod} provided that $\partial B$ generates the same sublattice of $\Z^2$ as $B$.
\end{theor}
In Section \ref{sec:iiss} we prove a stronger version of this fact (Theorem \ref{thintro} and further Theorem \ref{thh1}) for arbitrary dimension and tuples of non-equal support sets. In particular, this stronger version completely characterizes tuples $\tA_1=\ldots=\tA_n$ for which the Galois group $G_{\tA}$ is the expected wreath product.

The proof is based on a multidimensional version of Lemma \ref{ldim1nonred}. The multidimensional version of the trinomial deformation in the proof of this lemma becomes the most complicated part of the proof, requiring non-trivial considerations from the geometry of $A$-resultants and $A$-discriminants.

\subsection{Unexpected Galois groups of non-reduced systems.}\label{ssparsesysunexp}
In contrast to the one-dimensional case, starting from dimension 2, there do exist tuples $\tilde A$ such that $G_{\tA}$ is strictly smaller than the expected wreath product \eqref{eq:wreathprod}.

\begin{exa}\label{ex:smaller} Consider the non-reduced tuple $(Q,Q)$ with reduction $(P,P)$ where $P$ and $Q$ are pictured below. Since $\partial Q$ generates a strictly smaller sublattice than $Q$, Theorem \ref{thnonred0} does not ensure that the Galois group is the expected wreath product. Actually, we shall see that it is twice smaller in this particular case, although we do not know how to compute such non-expected Galois groups in general.
\begin{center}
\begingroup%
  \makeatletter%
  \providecommand\color[2][]{%
    \errmessage{(Inkscape) Color is used for the text in Inkscape, but the package 'color.sty' is not loaded}%
    \renewcommand\color[2][]{}%
  }%
  \providecommand\transparent[1]{%
    \errmessage{(Inkscape) Transparency is used (non-zero) for the text in Inkscape, but the package 'transparent.sty' is not loaded}%
    \renewcommand\transparent[1]{}%
  }%
  \providecommand\rotatebox[2]{#2}%
  \ifx\svgwidth\undefined%
    \setlength{\unitlength}{340.15748031bp}%
    \ifx\svgscale\undefined%
      \relax%
    \else%
      \setlength{\unitlength}{\unitlength * \real{\svgscale}}%
    \fi%
  \else%
    \setlength{\unitlength}{\svgwidth}%
  \fi%
  \global\let\svgwidth\undefined%
  \global\let\svgscale\undefined%
  \makeatother%
  \begin{picture}(1,0.20833333)%
    \put(0,0){\includegraphics[width=\unitlength,page=1]{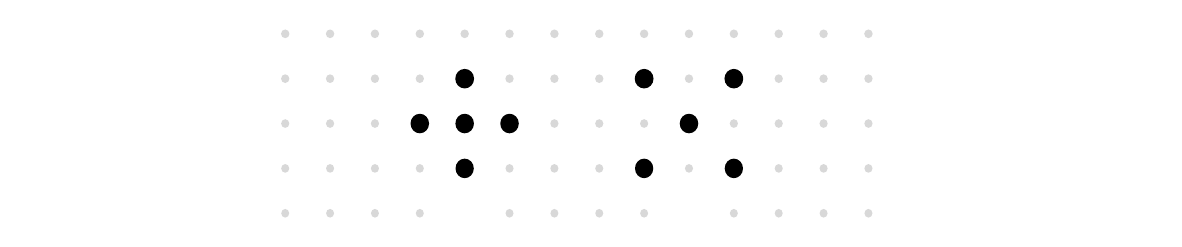}}%
    \put(0.37988885,0.004413){\color[rgb]{0,0,0}\makebox(0,0)[lb]{\smash{$P$}}}%
    \put(0.56833508,0.0056919){\color[rgb]{0,0,0}\makebox(0,0)[lb]{\smash{$Q$}}}%
  \end{picture}%
\endgroup%

\end{center}
\end{exa}

This can be seen 
as follows. The monodromy group consists of permutations of the eight roots along loops in the set $\C^Q\times\C^Q\setminus D$, where $\C^Q\times\C^Q$ is the space of systems of equations supported at $Q\subset\Z^2$, and $D$ is the bifurcation set (i.e. the set of all systems with more or less than eight roots). Thus, the monodromy group is generated by permutations, whose cyclic type is the same as for permutations along small loops around the components $D_i$ of the bifurcation set $D$.  

Applying the description of the irreducible components of the bifurcation set (Proposition 1.11/4.10 in the arXiv/journal version of \cite{adv} respectively)
to our case, we see that $D$ consists of 5 irreducible components: one component (the discriminant $D_0$) consists of systems with a root of multiplicity 2 (and hence two roots of multiplicity 2, because $Q$ generates an index 2 sublattice in $\Z^2$), and the other four components consist of systems with a root at one of the 4 one-dimensional orbits
of the toric variety $\CP^1\times\CP^1\supset\CC^2$. Thus, the permutation of roots along a small loop around $D_0$ consists of two disjoint transpositions. 

The other components $D_i$ of $D$ correspond to the edges $Q_i$ of the convex hull $Q$. By the same result from \cite{adv}, 
a generic system of equations from $D_i$ has several roots of multiplicity 1 in the complex torus and several roots of multiplicity $h$ at the $Q_i$-orbit of the $Q$-toric variety, where $h$ is the lattice distance from the line containing $Q_i$ to $Q\setminus Q_i$. In our case, $h=1$ for each of the four edges, so the permutations along small loops around the other four components of $D$ are trivial.

Thus the monodromy group is contained in $A_8\subset \sym_8$, while the wreath product $(\Z/2\Z)\wr \sym_4$ is not. Actually, one can manually check that the group $G_Q\subset \sym_8$ is the intersection of $(\Z/2\Z)\wr \sym_4$ with $A_8$.

\begin{rem}  In Section \ref{sec:iiss}, we  completely answer the following question for tuples $A_1=\ldots=A_n\subset\Z^n$:
\begin{equation}\label{eq:problem}
    \mbox{\it Determine whether the Galois group } G_A \; \mbox{\it equals the expected wreath product or not.}
\end{equation}
Taking this into account, one can distinguish three further key open questions in the study of Galois groups of general systems of polynomial equations.

1) If the answer to the question \eqref{eq:problem} is negative, how to compute the Galois group $G_A$ precisely? The question is open even for a system of two trinomials equations of two variables.

2) It is a purely combinatorial, but open and highly non-trivial problem to decide whether the results of this paper actually allow to answer the question \eqref{eq:problem} for every irreducible tuple. 
See Remark \ref{remincr2} for a precise combinatorial question.

3) We do not see a straightforward way to apply the technique from the present paper to the study of the Galois group for reducible tuples (such that $k$ of the sets can be shifted to the same $k$-dimensional sublattice). The Galois group is unknown even for a reducible system of two trinomial equations.

\end{rem}

\section{Inductive connectivity and the proof of Theorem \ref{mainth}}
In this section we introduce topological tools leading to the proof of Theorem \ref{topolth0} on the connectivity of inverse images under inductive coverings and, ultimately, Theorem \ref{mainth} on the monodromy of wreath enumerative problems.

Recall that in this paper all homology groups are defined over $\Z$. For the sake of brevity, we will therefore drop the coefficient group from the notation.  

Recall that a finite covering of connected spaces $\pi:X\to Y$ induces the natural pullback map $\pi^*:H_\bullet(Y)\to H_\bullet(X)$. At the level of chains, this map sends every singular simplex to the sum of its preimages. We refer to \cite{Hatcher} for more details on covering maps.

\begin{defin} An element of an Abelian group is \emph{primitive} if it cannot be written as the sum of $k>1$ copies of another element modulo torsion. 
\end{defin}
\subsection{Lifting connectivity in covering spaces}
\begin{lemma}\label{thindir2}
Let $\pi:X\to Y$ be a degree $d$ covering, and $V\subset Y$ be path-connected. If $V$ contains a singular 1-cycle $c$, 
such that $d\cdot c$ represents a primitive element in $\pi_*H_1(X)$,
then the preimage $U=\pi^{-1}(V)$ is path-connected.
\end{lemma}
\begin{rem} Since $d\cdot c=\pi_*\pi^* c$, we have:

-- the 1-cycle $d\cdot c$ always represents an element in $\pi_*H_1(X)$;

-- its primitivity in $\pi_*H_1(X)$ 
is equivalent to the primitivity of $\pi^*c$ in the group $H_1(X)/\ker \pi_*$ (and actually the latter restatement will be mostly used in what follows).
\end{rem}
{\it Proof.} 
Since every singular 1-cycle in a path-connected space is homological to a loop, we can assume with no loss in generality that the cycle $c$ in $V$ is a loop, pointed at some point $y$. 
Thus the lemma is reduced to the following special case. \hfill $\square$
\begin{lemma} \label{lconnect}
Let $\pi:X\to Y$ be a degree $d$ covering, let $c$ be a loop in $Y$, and denote by $b$ the singular 1-cycle $\pi^*c$ in $X$.
If $b$ represents a primitive element in $H_1(X)/\ker \pi_*$, then its support set $|b|$ is path-connected.
\end{lemma}
{\it Proof.} Assume towards the contradiction that $b$ is the sum of non-zero 1-cycles $b_1$ and $b_2$ with disjoint support sets, covering the loop $c$ with degrees $k_1$ and $k_2$ respectively. 
Denote ${\rm GCD}(k_1,k_2)$ by $k$, and $k_i/k$ by $k'_i$. Choose an integer decomposition $1=n_1k'_1+n_2k'_2$. Then, we have the identity
$$
\begin{array}{l}
  (k'_1+k'_2)\cdot (n_1,n_2) \; = \;  (1-n_2k'_2+n_1k'_2,\; n_2k'_1+1-k'_1n_1) \\
   \\
   \hspace{3.3cm} = \; (1,1) + (n_2-n_1)\cdot(-k'_2,k'_1) \; = \;  (1,1) \;{\rm mod}\; (-k'_2,k'_1).
\end{array}
$$
Taking the formal dot product of the vector $(b_1,b_2)$ with the two sides of this identity, we have 
\begin{equation}\label{eq:notprimitive}
    b=b_1+b_2=(k'_1+k'_2)\cdot(n_1b_1+n_2b_2)\;{\rm mod}\; g,
\end{equation}
where $g=k'_2\cdot b_1-k'_1\cdot b_2$. 
Since the projection $\pi_*(b_i)$ equals $k_i\cdot c$ as a singular 1-cycle (and all the more so as a homology 1-cycle), we have $k_2\cdot b_1-k_1\cdot b_2\in\ker\pi_*$. Dividing the latter element by $k$, we conclude that $g$ is a torsion element of the group $H_1(X)/\ker \pi_*$. Thus \eqref{eq:notprimitive} 
contradicts the primitivity of $b$ in $H_1(X)/\ker \pi_*$. 
 \hfill $\square$\medskip
 
Under few mild additional assumptions, we can strengthen the statement of Lemma \ref{thindir2} by replacing primitivity with a weaker property.
\begin{defin} An element of an Abelian group is \emph{weakly primitive} if it cannot be written as the sum of $k>1$ copies of another element. 
\end{defin}
\begin{lemma}\label{thindir2weak}
Let $\pi:X\to Y$ be a degree $d$ covering, and $V\subset Y$ be path-connected. Assume additionally one of the following:

- $d$ is prime;

- the covering $\pi$ is Galois. 

If $V$ contains a singular 1-cycle $c$ 
such that $d\cdot c$ represents a weakly primitive element in $\pi_*H_1(X)$, then the preimage $U=\pi^{-1}(V)$ is path-connected. 
\end{lemma}
As well as Lemma \ref{thindir2}, this one follows from the appropriate version of Lemma \ref{lconnect} below.
\begin{lemma} \label{lconnectweak}
Let $\pi:X\to Y$ be a degree $d$ covering, let $c$ be a loop in $Y$ and denote by $b$ the singular 1-cycle $\pi^*c$ in $X$.
Assume additionally one of the following:

- $d$ is prime;

- the covering $\pi$ is Galois.

If $b$ represents a weakly primitive element in $H_1(X)/\ker \pi_*$, then its support set $|b|$ is path-connected.
\end{lemma}
{\it Proof.} If $d$ is prime, then $k=1$ and thus $g\in\ker\pi_*$ in the proof of the original Lemma \ref{lconnect}.

If $\pi$ is Galois, then its group of deck transformations acts transitively. In particular, it acts transitively on the connected components $s_1,\ldots,s_k$ of $|b|$ and on the summands of the respective decomposition $b=b_1+\ldots+b_k,\, |b_i|=s_i$. Thus the push-forward $\pi_*b_i$ does not depend on $i$, and we have $b_i-b_j\in\ker\pi_*$ for all $i$ and $j$. This implies $b=k\cdot b_1\;{\rm mod}\;\ker\pi_*$. In particular, we have $k=1$, because $b$ is weakly primitive in $H_1(X)/\ker \pi_*$.
\hfill $\square$\medskip

In the target case of algebraic groups, the observations of this section simplify as follows.
\begin{sledst}
Let $\pi:X\to Y$ be a covering of connected Lie groups.
Let $V\subset Y$ be path-connected. 
If $V$ contains a singular 1-cycle $c$ whose preimage $\pi^*(c)$ represents a weakly primitive element in $H_1(X)$, then the preimage $U=\pi^{-1}(V)$ is path-connected.
\end{sledst} 
{\it Proof.} At the level of fundamental groups, the map $\pi_1(X)\to\pi_1(Y)$ is injective for any connected covering $\pi:X\to Y$. Moreover, since the fundamental groups of Lie groups are commutative, we have $\pi_1(Y)=H_1(Y)$, and the diagram \eqref{eq:comdiag}
ensures that $\pi_*:H_1(X)\to H_1(Y)$ is injective as well. In this case $\ker \pi_*$ is trivial, and the statement follows from Lemma \ref{thindir2weak}, because a covering of Lie groups is Galois. 
\hfill $\square$\medskip

\begin{rem}
In the present work, $X$ and $Y$ will be complex tori, so the absence of torsion in their homology identifies primitivity with weak primitivity. However, for general algebraic groups, the torsion (and thus the difference between the two notions of primitivity) is non-trivial and important for their geometry.
\end{rem}

\subsection{Inductive connectivity.}

Our aim here is to create a context in which we could apply the results of the preceding subsection to a chain of covering maps $\ldots\to X_2\to X_1\to Y$ to prove by induction on $k$ that the preimage of a certain path-connected $V\subset Y$ in $X_k$ is path-connected. In this way we shall prove Theorem \ref{topolth0}. 

\begin{theor}\label{thind}
Let $\pi:X\to Y$ be an inductive covering, and let $U\subset X$ be the preimage of a subset $V\subset Y$, then the image of $\h{U}$ in $\h{X}$ is the intersection of $\h{X}\subset\h{Y}$ with the image of $\h{V}$ in $\h{Y}$. 

\end{theor}

\begin{proof} 
Observe first that we can treat each path-connected component of ${V}$ separately and therefore assume that ${V}$ is path-connected. In view of the commutative diagram
$$
\begin{tikzcd}
H_1({U}) \arrow[r,symbol=\longrightarrow] \arrow[d,symbol=\longrightarrow] & H_1(X) \arrow[d,symbol=\hookrightarrow] \\
H_1({V}) \arrow[r,symbol=\longrightarrow] & H_1(Y) 
\end{tikzcd}
$$
the image of $\h{{U}}$ in $\h{X}$ is contained in the intersection of $\h{X}\subset\h{Y}$ with the image of $\h{{V}}$ in $\h{Y}$. Since ${V}$ is path-connected by assumption, the containment in the opposite direction is equivalent to showing that, for any loop $c$ in $X$ pointed at $x\in {U}$ such that  the homology class $\left[\pi_*(c)\right]$ belongs to the image of $\h{{V}}$ in $\h{Y}$, the homology class $\left[c\right]$ belongs to the image of $\h{{U}}$ in $\h{X}$.

For $c$ as above, the loop $\pi_*(c)$ in $Y$ is homotopic to the product of a loop $b$ in ${V}$ (representing the same element in the image of $\h{{V}}$) and a loop $h$ representing an element of the commutator $\pi'_1(Y)$. By the inductivity of the covering, $h$ is the image of a loop $g$ representing an element of the commutator $\pi'_1(X)$. The homotopy between the loops $\pi_*(c)h^{-1}$ and $b$ lifts to the homotopy between the loop $cg^{-1}$ and a certain loop $\gamma$ in ${U}$ covering $b$. Since $g$ belongs to $\pi'_1(X)$, it follows that $\left[c\right]= \left[cg^{-1}\right]= \left[\gamma\right]$ belongs to the image of $\h{{U}}$ in $\h{X}$. 
\end{proof}

In particular, for every subgroup $L\subset \h{X}$, this theorem implies the following. 
\begin{sledst}\label{corH+L}
Let $\pi:X\to Y$ be an inductive covering, and let the subset  ${U}\xhookrightarrow{i}X$ be the preimage of a subset  ${V}\xhookrightarrow{j}Y$.
If $j_*\h{{V}}+\pi_* L$ generates $\h{Y}$, 
then $i_*\h{{U}}+L$ generates $\h{X}$.
\begin{proof}
By assumption, any element $x \in \pi_*\h{X}$ can be written as a sum $x=v+\ell$ with $v\in j_*\h{{V}}$ and $\ell \in \pi_*L$. In particular, we have that $v=x-\ell$ is an element of $\pi_*\h{X}$. By Theorem \ref{thind}, $v$ is an element in $i_*\big(\pi_*\h{{U}}\big)$. It follows that any element in $\h{X}$ is the sum of an element in $i_*\h{{U}}$ with an element in $L$.
\end{proof}

\end{sledst}
This motivates the following notion.
\begin{defin}\label{definductive} A path-connected $V\subset Y$ is said to be \emph{$L$-inductively connected} (or just \emph{inductively connected} for $L=0$) for a subgroup $L\subset \h{Y}$, if the image of $\h{V}$ together with $L$ generates $\h{Y}$.
\end{defin}

\subsection{Proof of Theorem \ref{topolth0}.}
We have reached the goal stated at the beginning of the preceding subsection.

\begin{theor}\label{corolmain} 1) Assume that $\pi:X\to Y$ is a strongly inductive covering of finite degree,

and that $V\xhookrightarrow{j}Y$ is $(\pi_* L)$-inductively connected for some subgroup $L\subset \h{X}$. 

Then, the preimage $U:=\pi^{-1}(V)\xhookrightarrow{i}X$ is $L$-inductively connected.

2) In particular, if $j_*\h{V}+\pi_*(\h{X})=\h{Y}$, then $U$ is path-connected.

3) Conversely, if $j_*\h{V}+\pi_* (\h{X})\ne \h{Y}$, then $U$ is  path-disconnected. 
\end{theor}
\begin{proof} 
We first prove 1) for a covering of prime degree $p>1$. Since $\pi_* : \h{X} \rightarrow \h{Y}$ is injective and restricts to an isomorphism on the torsion part, we can find two minimal sets of generators $a_1,\dots, a_k$ and  $b_1,\dots, b_k$ of $\h{X}$ and $\h{Y}$ respectively such that $a_1$ and $b_1$ are not torsion elements and such that $\pi_*$ is given by 
$$\pi_*(a_1)=p\cdot b_1  \text{ and } \pi_*(a_j)= b_j \text{ for } j\geq 2.$$

In particular, the following conditions are equivalent for a singular $1$-cycle $c$ in $V$:

-- $\left[p\cdot c\right]$ is a primitive element in $\pi_*\h{X}$; 

-- $[c]$ is primitive and not contained in $\pi_*\h{X}$.

To see this, write $[c]=(\lambda_1,\cdots,\lambda_k)$ in the coordinates provided by $b_1,\cdots,b_k$ and assume that $b_1, \cdots,b_\ell$  are the non-torsion elements among them ($\ell\leq k$). On the one hand, the class $[c]$ is primitive and not in $\pi_*\h{X}$ if and only if the vector $(\lambda_1,\cdots, \lambda_\ell)$ is primitive and $\lambda_1$ is not divisible by $p$. On the other hand, the class $[p\cdot c]$ is primitive in $\pi_*\h{X}$ if and only if the vector $(\lambda_1, p\lambda_2, \cdots, p \lambda_\ell)$ is primitive. The two properties are therefore equivalent. 

Since $j_*\h{V}+\pi_*(L)$ generates $\h{Y}$, and $\pi_*(L)\subset \pi_*(\h{X})$, there exists at least one cycle $c$ that satisfies one of the two equivalent conditions above. It follows from Lemma \ref{thindir2} that $U$ is path-connected and from Corollary \ref{corH+L} that $i_*\h{U}+ L$ generates $\h{X}$. 
The subvariety $U\subset X$ is therefore $L$-inductively connected.

Now we can prove 1) for a covering of an arbitrary degree.

We first show that $\pi$ can be written as a composition of strongly inductive coverings whose respective degrees are prime numbers. Indeed, the inclusion of lattices $H_1(X)/H^T_1(X)\subset H_1(Y)/H^T_1(Y)$ admits the Smith normal form and thus extends to an increasing filtration of lattices $L_i$, such that the index of every two consecutive lattices $L_i$ and $L_{i+1}$ is prime. Since path-connected coverings $\tilde \pi : \widetilde X\to Y$ such that $\pi$ factorizes through $\tilde \pi$ are in correspondence with the subgroups $G<\pi_1(Y)$ containing $\pi_*(\pi_1(X))$ (see \cite[Theorem 1.38]{Hatcher}), the filtration of lattices $L_i$ gives rise to a decomposition of the covering $\pi$ into a sequence of spaces $X_i$ and covering spaces between them. These are strongly inductive by Example \ref{exastrongind}.4. The degrees of these coverings are prime, because the degree of a strongly inductive $\tilde\pi$, which is given by the index of $\pi_*(\pi_1(X))$ in $\pi_1(Y)$, equals the index of $H_1(X)$ in $H_1(Y)$, and further the index of $H_1(X)/H^T_1(X)$ in $H_1(Y)/H^T_1(Y)$.

Now, for every intermediate covering $X\xrightarrow{\pi_i} X_i\xrightarrow{\tilde\pi_i} Y$ we can deduce by induction on $i$ that the set $\tilde\pi_i^{-1}(V)\subset X_i$ is $\pi_{i*}(L)$-inductively connected. The step of the induction is the statement 1) of the theorem for the prime degree covering $X_{i+1}\to X_i$, which we have already proved.

The statement 2) is a special case of 1) with $L=\h{X}$. For the statement 3), 

the subgroup $j_*\h{V}+\pi_*(\h{X})$ is contained in a strict subgroup of $\h{Y}$. In particular, the subgroup generated by $\pi_1(V)$ and $\pi_*(\pi_1(X))$ in $\pi_1(Y)$ is contained in a strict subgroup $G\subset \pi_1(Y)$. Therefore, the covering $\pi$ factors through the strongly inductive covering $\tilde \pi : \widetilde X \rightarrow Y$ associated to $G$ with the property that $j_*\h{V}+\pi_*(\h{X}) \subset \tilde \pi_*(\h{\widetilde X})$.

Assuming to the contradiction that $\widetilde{U}:=\tilde\pi^{-1}(V)$ is path-connected, we can connect two of the preimages of $y\in V$ through $\widetilde{U}$ with a path $\gamma$. We claim that the loop $\tilde{\pi}(\gamma)$ represents a cycle in $\h{V}$ outside $\tilde\pi_* (\h{\widetilde X})$, leading to a contradiction. Thus, the set $\tilde\pi^{-1}(V)$ is not path-connected and so is $\pi^{-1}(V)$, as $\pi$ factors through $\tilde\pi$.

It remains to prove the above claim. Assume to the contradiction that $\tilde{\pi}(\gamma)$ represents a cycle in $\tilde\pi_* (\h{\widetilde X})$. Then, the loop $\tilde{\pi}(\gamma)$ is homotopic to the product $a\cdot b$ of a loop $a \in \tilde \pi(\pi_1(X))$ with a commutator $b\in \pi_1'(Y)$, both based at $y$. Since $a \in \tilde \pi(\pi_1(X))$ and $\tilde \pi$ is inductive, we can lift $a$ and $b$ to loops $c$ and $d$ based at the starting point $x$ of $\gamma$. Therefore, the homotopy between $a\cdot b$ and  $\tilde{\pi}(\gamma)$ lifts to an homotopy between $c\cdot d$ and a loop based at $x$ and covering $\tilde{\pi}(\gamma)$ once. Since there is only one lift of $\tilde{\pi}(\gamma)$ starting at $x$, it follows that $\gamma$ is a loop. This is a contradiction.
\end{proof}

\subsection{Powers of enumerative problems.} 
We shall prove Theorem \ref{mainth} by applying Theorem \ref{topolth0} to a  \emph{fiber power} of the given enumerative problem $U\subset T\times C$: let $U^{(k)}$ be the space of tuples $(x_1,\ldots,x_k,f)\in T^k\times C$ such that $x_1,\ldots,x_k$ are pairwise distinct solutions in $c^{-1}(f)$. If the projection $U\to C$ is a $d$-fold covering, then $U^{(k)}$ is a smooth algebraic set in $T^k\times C$, and the projection $U^{(k)}\to C$ is 
a $\frac{d!}{(d-k)!}$-fold covering. We need the following well known observation (the case $k=2$ is especially important, see e.g. \cite{Harris} and \cite{schsot}).

\begin{observ}\label{exatransit}
1) The monodromy group of $U\to C$ is  $k$-transitive (i.e. capable of sending any given $k$-tuple of solutions to any other given $k$-tuple) if and only if the monodromy group of $U^{(k)}\to C$ is transitive, i.e. $U^{(k)}$ is connected.

2) In particular, the monodromy group of $U\to C$ equals the symmetric group $\sym_d$ if and only if the monodromy group of $U^{(d)}\to C$ is transitive, i.e. $U^{(d)}$ is connected.
\end{observ}

\smallskip
\subsection{Proof of Theorem \ref{mainth}.}\label{proofofmainth}

The inclusion $\tG \subset D\wr G$ makes the sought equality $\tG =D\wr G$ equivalent to the numerical one $|\tG|=|D\wr G|$ and, furthermore to 
\begin{equation}\label{eq:numerical}
    \big|(D\wr \sym_d)/\tG\big|=|\sym_d/G|,
\end{equation}
since the right hand side equals $\big|(D\wr \sym_d)/(D\wr G)\big|$.

In order to prove \eqref{eq:numerical}, we will interpret the two sides of this equality as the numbers of connected components of certain enumerative problems, derived from the initial one.
Namely, denote the projections of $U\subset T\times C$ to $T$ and $C$ by $t$ and $c$ respectively, and consider the following objects:

-- the smooth algebraic set $V:=U^{(d)}\subset T^d\times C$, that is the set of all tuples $(x_1,\ldots,x_d,f)$ such that $x_1,\ldots,x_d$ are the (arbitrarily ordered) points of the fiber $c^{-1}(f)$. 

-- the covering $\pi:\tT^d\times C\to T^d\times C$;

-- the preimage $\tV=\pi^{-1}(V)$, that is the set of all tuples $(y_1,\ldots,y_d,f)$ such that $y_1,\ldots,y_d$ are $d$ points of the fiber $\phi^{-1}\circ c^{-1}(f)$ such that $\phi(y_1),\ldots,\phi(y_d)$ are pairwise distinct.

Choosing a base point $f\in C$ with fiber $S:=c^{-1}(f)$, a ranking of this fiber $\delta:S\to\{1,2,\ldots,d\}$ defines 

-- the base point $f_\delta:=(\delta^{-1}(1),\ldots,\delta^{-1}(d),f)$ in $V$, 

-- the component of $V$ containing $f_\delta$ (we denote it by $V_\delta$),

-- the natural action of the symmetric group $\sym_d$ on $S$, its induced action on the fiber of the covering $V\to C$, and the induced embedding of the monodromy group $G\subset \sym_d$.

The regular action of the group $\sym_d$ on a fiber of the covering $V\to C$ restricts to the natural action of the monodromy group $G\subset \sym_d$ and the orbits of $G$ are in correspondence with the components of $V$. This is because, by the definition of the action, two points of the fiber are in the same orbit if and only if they can be connected with a path in $V$. In particular, $V$ has $|\sym_d/G|$ components.

Similarly, the regular action of $D\wr \sym_d$ on a fiber of the covering $\tV\to C$ restricts to the natural action of the monodromy group $\tG\subset D\wr \sym_d$, and the orbits of the latter are in correspondence with the components of $\tV$. In particular, $V$ has $\big|(D\wr \sym_d)/\tG\big|$ components.

As a result, the equality \eqref{eq:numerical} is equivalent to the fact that every component $V_\delta$ of $V$ is covered with a unique component of  $\tV$. 
The latter is equivalent to \eqref{eq:solutionlattice} by Theorem \ref{corolmain}, applied to the covering $\pi$ and the set $V_\delta$. This application is justified by the following observations:

-- The covering $\pi$ is (strongly) inductive if and only if the covering $\phi$ is;

-- Denoting the projections of $V$ to $T^d$ and $C$ by $t$ and $c$ respectively, $c_*$ surjectively maps $\pi_1(V_\delta,f_\delta)$ to the kernel of the monodromy map $\pi_1(C)\rightarrow \sym(S)$. Thus the map $\gamma\mapsto\gamma_\delta$ from the definition of the solution lattice decomposes into lifting the loop $\gamma$ to $\pi_1(V_\delta,f_\delta)$ and then mapping the homology class of the lifted loop with  $t_*:H_1(V_\delta)\to H_1(T^d)$. In particular, 
the solution lattice equals $t_*\h{V_\delta}$.

As a result, the equality \eqref{eq:solutionlattice} is equivalent to the equality $\h{T^d\times C}^d=\h{V_\delta}+\pi_*\h{\tT^d\times C}$ that appears in Theorem \ref{corolmain}.

\section{Galois groups of systems of sparse polynomial equations.}\label{sec:iiss}

In this section, we state and prove a generalization of Theorem \ref{thnonred0} that in particular

-- completely characterizes tuples $A_1=\ldots=A_n\subset\Z^n$ such that the Galois group $G_A$ of the corresponding system of sparse equations equals the expected wreath product;

-- extends this result to tuples of support sets that are not equal but similar in a sense (see the notion of analogous sets below).

As in Subsections \ref{ssparsesys}--\ref{ssparsesysunexp}, the tuple of finite sets $A:=(A_1,\ldots,A_n)$ in the character lattice $\Z^n$ of the complex torus $T:=\CC^n$ is assumed to be reduced and irreducible, see Definition \ref{def:redirred}. The mixed volume of the convex hulls of $A_1,\ldots,A_n$ is denoted by $d$.

\subsection{Analogous systems of equations}

For a linear function $\gamma:\Z^n\to\Z$ and a finite set $P\subset\Z^n$, let $P^\gamma\subset P$ be the set of all points where $\gamma$ attains its maximal value on $P$.
\begin{defin} \label{defanalogous}
We say that the sets $A_1,\ldots,A_n$ are \emph{analogous} if, for every $\gamma\in(\Z^n)^*$, there exists a vector subspace $V_\gamma\subset\R^n$ such that for all $i=1,\dots,n$, the minimal affine subspace containing $A_i^\gamma$ is a shifted copy of $V_\gamma$. 

\end{defin}
Equivalently, the sets $A_1,\ldots,A_n$ are analogous if the convex hulls of  $A_1,\ldots,A_n$ share the same dual fan, see \cite[Section 1.5]{Ful}.
\begin{exa} If the convex hulls of $A_1,\ldots,A_n$ are equal, or more generally homothetic, then $A_1,\ldots,A_n$ are analogous.
\end{exa}
\begin{rem} A tuple of analogous sets of full dimension is always irreducible. A tuple is analogous if and only if its reduction is analogous, see Remark \ref{rem:reduction}.
\end{rem}
Let $\mathcal{G}_A\subset (\Z^n)^*$ be the (finite) set of all primitive $\gamma$ such that $V_\gamma$ is a hyperplane, and let $d_\gamma$ be the index in $V_\gamma$ of the minimal sublattice to which each of $A_1^\gamma,\ldots,A_n^\gamma$ can be shifted.
\begin{defin} Let $\tilde A$ be an analogous tuple in a lattice $\tN$ with reduction $A$ (see Subsection \ref{ssparsesysexp}) given by a linear embedding $L:N\to\tN$ and denote by $L^*$ the dual embedding. We say that $\tilde A$ is \emph{ample} if the vectors $d_\gamma\cdot\gamma \in N^*$, $\gamma\in\mathcal{G}_{A}$, together with the lattice $\im(L^*)$ generate $N^*$. 
\end{defin}
In particular, a reduced analogous tuple $A$ is always ample, since $\im(L^*)=N^*$. Note that the property of being ample does not depend on the choice of the reduction.
\begin{theor} \label{thintro} Let $\tilde A:= (\tilde A_1,\ldots,\tilde A_n)$ be a tuple of finite sets in a lattice $\tN$ such that $0\in A_i$ and denote by $\Lambda\subset \tN$ the sublattice generated by the sets $\tilde A_i$. Denote also by $\tilde d$ the mixed volume of the convex hulls of these sets and let $d:=\tilde d/|\Z^n/\Lambda|$.

1)  Assume that $\tilde A_1,\ldots,\tilde A_n$ are analogous. Then, the monodromy group $G_{\tilde A}$ of the system of equations with indeterminate coefficients supported at $\tilde A$ is isomorphic to $(\tN/\Lambda) \wr \sym_d$ if $\tilde A$ is ample, and is strictly smaller otherwise.

2) Assume that $\tilde A$ has a reduction $A:=(A_1,\ldots,A_n)$ such that every $A_i$ is contained in the positive quadrant $\Z^n_{\geqslant 0}$ of $N\simeq\Z^n$ and contains the vertices of the standard simplex (i.e. $\C^{A_i}$ consists of non-Laurent polynomials and contains the space of affine linear functions). Then, the monodromy group $G_{\tilde A}$ is isomorphic to $(\tN/\Lambda) \wr \sym_d$.
\end{theor}

The proof essentially repeats Theorem \ref{thdim1nonred}: we prove the multidimensional version of Lemma \ref{ldim1nonred} (see Theorem \ref{thh1} below) claiming that the solution lattice of the $A$-enumerative problem is sufficiently large, and then apply Theorem \ref{mainth}. However, the trinomial deformation of a given polynomial equation from the proof of Lemma \ref{ldim1nonred} becomes drastically more complicated in the multidimensional setting, and its construction occupies most of this section. It is based on geometry of $A$-resultants and $A$-discriminants.

\smallskip
\subsection{Resultants.}\label{sec:res} 
The 
first homology group $H:=\h{T}$ is the lattice dual to $\Z^n$: the composition of a loop $S^1\to\CC^n$ representing a cycle $\gamma\in H$ and a monomial $m:\CC^n\to\CC^1$ is a map $S^1\to \CC^1$, and its class $r\in\pi_1\CC^1=\Z$ defines the natural non-degenerate pairing $H\times\Z^n\to\Z,\,(\gamma,  m) \mapsto \gamma \cdot  m=r$. 

Every $\gamma\in H$, considered as a linear function on $\Z^n$ attains its maximum on $A_i$ at some subset that we denote by $A^\gamma_i$. For short, the tuples $(A_1,\ldots,A_n)$ and $(A^\gamma_1,\ldots,A^\gamma_n)$ will be denoted by $A$ and $A^\gamma$ respectively, and the spaces of systems of equations $\C^{A_1}\oplus\ldots\oplus\C^{A_n}$ and $\C^{A^\gamma_1}\oplus\ldots\oplus\C^{A^\gamma_n}$ supported at these tuples  -- by $\C^A$ and $\C^{A^\gamma}$.

The \emph{reduced resultant} $R^{red}_{A^\gamma}$ is the closure of the set of all tuples $g=(g_1,\ldots,g_n)\in\C^{A^\gamma}$ such that the system $g_1(x)=\ldots=g_n(x)=0$ has a root $x\in\CC^n$. All $A^\gamma_i$ by definition can be shifted to the hyperplane $\ker \gamma$ (where $\gamma$ is considered as a linear function on $\Z^n$), so the set of solutions of $g_j=0$
is invariant under the action of the 1-dimensional subtorus $T_\gamma\subset\CC^n$ whose homology embeds in $H$ as $\Z\cdot\gamma$. 

\begin{defin} A primitive vector $\gamma\in H$ is said to be \emph{essential}, if the tuple $A^\gamma$ does not contain $k\geqslant 2$ sets that can be shifted to the same $(k-2)$-dimensional plane. The set of essential vectors will be denoted by $\mathcal{G}$.

\end{defin}
\begin{rem}
The set $\mathcal{G}$ is finite: in particular, it is contained in the set of all primitive exterior normal vectors to the facets of the convex hull of $A_1+\ldots+A_n$.
\end{rem}

The resultant $R^{red}_{A^\gamma}$ is an irreducible hypersurface (see \cite{S94}) if and only if $\gamma$ is proportional to an essential vector. Then the dimension count shows that, for a generic tuple $g\in R^{red}_{A^\gamma}$, the zero locus $\{g=0\}$ is one-dimensional. Thus its quotient by the torus $T_\gamma$ is a finite set, whose cardinality will be denoted by $d_\gamma$. This number should be regarded as a natural multiplicity of the resultant $R^{red}_{A^\gamma}$, and will be explicitly computed in Theorem \ref{thm:degres} below. 

\begin{defin}[\cite{dcg}] 1) If $\gamma$ is essential, then the \emph{algebraic resultant} $R_{A^\gamma}$ is defined as $F^{d_\gamma}$, where $F$ is the irreducible polynomial defining the hypersurface $R^{red}_{A^\gamma}$, and the number $d_\gamma$ is the \emph{resultant multiplicity}, defined above.

2) If $\gamma$ is not proportional to an essential vector, then by definition the algebraic resultant $R_{A^\gamma}$ is 1, and its multiplicity $d_\gamma$ equals 0.
\end{defin}

The computation of the resultant multiplicty $d_\gamma$ for an essential vector $\gamma$ is based on the following observation from \cite{S94}. The set $I$ of subsets $K\subset\{1,\ldots,n\}$ such that all $A^\gamma_i,\, i\in K$, can be shifted to the same $(|K|-1)$-dimensional sublattice, has a unique element which is minimal with respect to inclusion. Indeed, the set $I$ is not empty since it contains $\{1,\ldots,n\}$, and the intersection of two elements of $I$ necessarily belongs to $I$, otherwise their union would certify that $\gamma$ is not essential (see \cite{S94} for details). The minimal element in $I$ will be denoted by $K_\gamma$, and the minimal sublattice to which all $A^\gamma_i,\, i\in K_\gamma$, can be shifted, will be denoted by $L_\gamma\subset\Z^n$. 

Let $d'_\gamma$ be the index of the lattice $L_\gamma$ in its saturation $\bar L_\gamma$. The images of the sets $A^\gamma_i,\, i\notin K_\gamma$, under the projection $\Z^n\to\Z^n/\bar L_\gamma$ are $n-|K_\gamma|+1$ sets in a lattice of dimension $n-|K_\gamma|+1$. Thus, the lattice mixed volume of the convex hulls of these images makes sense and is denoted by $d''_\gamma$. 
\begin{theor}[Theorem 2.23 in \cite{mian}]
\label{thm:degres} 1) The resultant multiplicity $d_\gamma$ equals $d'_\gamma\cdot d''_\gamma$. 

2) For a generic tuple $g\in R^{red}_{A^\gamma}$, among the differentials $dg_1,\ldots,dg_n$ at a point of the zero locus $\{g=0\}$, the only linearly dependent subtuple is $dg_i,\, i\in K_\gamma$, and its corank is $1$ (i.e. they satisfy a unique non-trivial linear relation).
\end{theor}

\begin{defin} The tuple $(E_1,\ldots,E_n)$ such that $E_k:=\varnothing$ for $k\notin K_\gamma$ and $E_k:=A^\gamma_k$ otherwise, is said to be the \emph{essential tuple} defined by an essential vector $\gamma$.

We denote by $A^\gamma_{ess}$ the essential tuple defined by $\gamma\in\mathcal{G}$, denote the set of all essential tuples by $\mathcal{E}$, and denote by $\mathcal{E}_0\subset\mathcal{E}$ the set of tuples $(E_1,\ldots,E_n)\in\mathcal{E}$ such that the convex hull of  every $E_i$ has dimension $(n-1)$.
\end{defin}
Note that different vectors $\gamma\in\mathcal{G}$ may give rise to the same essential tuple.
\begin{exa} For $A_1$ and $A_2$ as on the picture below, both $(-1,0)$ and $(0,-1)$ belong to $\mathcal{G}$ and give the same essential tuple $(\varnothing, $ the bottom-left point of $A_2)$.
\end{exa}
\begin{center}
\begingroup%
  \makeatletter%
  \providecommand\color[2][]{%
    \errmessage{(Inkscape) Color is used for the text in Inkscape, but the package 'color.sty' is not loaded}%
    \renewcommand\color[2][]{}%
  }%
  \providecommand\transparent[1]{%
    \errmessage{(Inkscape) Transparency is used (non-zero) for the text in Inkscape, but the package 'transparent.sty' is not loaded}%
    \renewcommand\transparent[1]{}%
  }%
  \providecommand\rotatebox[2]{#2}%
  \ifx\svgwidth\undefined%
    \setlength{\unitlength}{340.15748031bp}%
    \ifx\svgscale\undefined%
      \relax%
    \else%
      \setlength{\unitlength}{\unitlength * \real{\svgscale}}%
    \fi%
  \else%
    \setlength{\unitlength}{\svgwidth}%
  \fi%
  \global\let\svgwidth\undefined%
  \global\let\svgscale\undefined%
  \makeatother%
  \begin{picture}(1,0.20833333)%
    \put(0,0){\includegraphics[width=\unitlength,page=1]{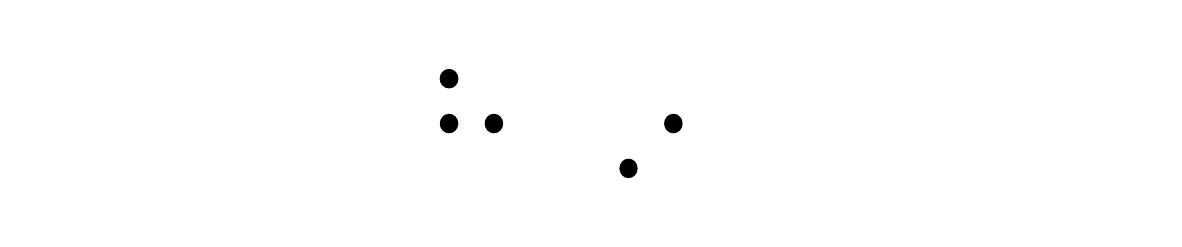}}%
    \put(0.36665969,0.0220518){\color[rgb]{0,0,0}\makebox(0,0)[lb]{\smash{$A_1$}}}%
    \put(0.55510584,0.02333071){\color[rgb]{0,0,0}\makebox(0,0)[lb]{\smash{$A_2$}}}%
    \put(0,0){\includegraphics[width=\unitlength,page=2]{noness.pdf}}%
  \end{picture}%
\endgroup%

\end{center}

\begin{rem}\label{remanalogous}
1) A vector $\gamma \in \mathcal{G}$ satisfies $A^\gamma_{ess}\in \mathcal{E}_0$ if and only if there exists a facet in every $A_i$ whose primitive normal exterior vector is $\gamma$.

2) The map $\mathcal{G}\to\mathcal{E},\, \gamma\mapsto A^\gamma_{ess}$, is one to one over $\mathcal{E}_0$, i.e. every essential tuple $E\in\mathcal{E}_0$ is defined by a unique $\gamma\in\mathcal{G}$, which we denote by $\gamma_E$.

3) If $A_1,\ldots,A_n$ are analogous (Definition \ref{defanalogous}), then $\mathcal{E}_0=\mathcal{E}$, and $A^\gamma_{ess}=A^\gamma$. As a consequence, 
$\mathcal{G}$ and $\mathcal{E}$ are in one to one correspondence.
\end{rem}

By a harmless abuse of notation, we denote the lift of the algebraic discriminant $R_{A^\gamma}$ under the natural forgetful projection $\C^A\to\C^{A^\gamma}$ by the same symbol $R_{A^\gamma}$. Let $\mathcal{R}$ be the set of all hypersurfaces of the form $R_{A^\gamma}=0$ in $\C^A$.
\begin{rem}\label{remanalogous2}
The terminology ``essential tuple'' is motivated by the fact that, by construction, two equations $R_{A^\gamma}=0$ and $R_{A^{\gamma'}}=0$ define the same set in $\C^A$ if and only if $\gamma$ and $\gamma'$ define the same essential tuple. Thus, $\mathcal{R}$ is in one to one correspondence with the set of essential tuples $\mathcal{E}$ (and, by Remark \ref{remanalogous}, also with $\mathcal{G}$ when $A$ is analogous). For $E\in\mathcal{E}$, denote the corresponding resultant set in $\C^A$ by $R^{red}_E\in\mathcal{R}$.
\end{rem}

\smallskip
\subsection{The main result.} 
We consider a reduced irreducible tuple $A:=(A_1,\ldots,A_n)$ of finite sets in $\Z^n$. The general system of equations supported at this tuple has Galois group $\sym_d$ (where $d$ is the generic number of roots of the system, i.e. the lattice mixed volume of the convex hulls of $A_i$). 

However, if this system of equations undergoes a monomial change of coordinates corresponding to a proper sublattice $L\subset H$, then the Galois group of the resulting non-reduced system is not symmetric.
In order to analyse this Galois group $G_{L,A}$ using Theorem \ref{mainth}, we should check whether 
\begin{equation}\label{eq:inductiveconn}
    H^{\oplus d}=L^{\oplus d}+H_\delta,
\end{equation}
or, in other words, whether the solution space of the initial system of equations is $L^{\oplus d}$-inductively connected.

Recall the definition of the solution lattice $H_\delta$ in this context. Pick a generic system of equations $f_\circ$ in the space $\C^A$, and order its roots: $\delta:\{x\,|\, f_\circ(x)=0\}\to\{1,\ldots,d\}$. Every loop $\alpha$ in $\C^A$ pointed at $f_\circ$ defines a permutation of the roots. If this permutation is trivial, then, as the system of equations travels along the loop, its $i$-th root (with respect to the order $\delta$) travels along a loop $\alpha_i$ in the torus $T$, and the homology classes $\tilde\alpha_i$ of these loops define an element $\tilde\alpha_\delta\in H^{\oplus d}$. The set of all such elements is the solution lattice $H_\delta$.

We shall prove the following criterion of whether the equality \eqref{eq:inductiveconn} holds for given $A$ and $L$. For an essential tuple $B\in\mathcal{E}$, define $\mathcal{G}_B$ to be the set of all vectors $\gamma\in\mathcal{G}$ supporting this essential tuple: $A_{ess}^\gamma=B$. 
\begin{theor}\label{thh1} 
1) Assume that 
the vectors $\sum_{\gamma\in\mathcal{G}_B} d_\gamma\cdot\gamma$ over all $B\in\mathcal{E}$ together with $L$ do not generate the lattice $H$. Then \eqref{eq:inductiveconn} is not satisfied, and the Galois group $G_{L,A}$ is strictly smaller than the wreath product $(H/L)\wr\sym_d$.

2) Assume that 
the vectors $d_{\gamma_B}\cdot\gamma_B$ over all $B\in\mathcal{E}_0$ together with $L$ generate the lattice $H$. Then $\eqref{eq:inductiveconn}$ is satisfied, and the Galois group $G_{L,A}$ is isomorphic to the wreath product $(H/L)\wr\sym_d$.

\end{theor}
\begin{rem}\label{remincr} 1) 
If the tuple $A$ is analogous, then Remark \ref{remanalogous} assures that the set of vectors from \ref{thh1}.1) coincides with that of \ref{thh1}.2), i.e. $\sum_{\gamma\in\mathcal{G}_B} d_\gamma\cdot\gamma=d_{\gamma_B}\cdot\gamma_B$. Thus, the theorem completely characterizes analogous tuples satisfying the equality \eqref{eq:inductiveconn}.

2) In general, we have the following three increasing 
classes of tuples, which coincide for analogous tuples:

(a) Tuples $A$, such that the vectors  $d_{\gamma_B}\cdot{\gamma_B}$ over all $B\in\mathcal{E}_0$, generate the lattice;

(b) Tuples satisfying the equality \eqref{eq:inductiveconn};

(c) Tuples $A$, such that the vectors $\sum_{\gamma\in\mathcal{G}_B} d_\gamma\cdot\gamma$ over all $B\in\mathcal{E}$ generate the lattice.

We expect that, for general reduced irreducible tuples, (b) is strictly larger than (a). Regarding the comparison of (b) and (c), see the subsequent Remark \ref{remincr2}.
\end{rem}
\smallskip
\subsection{Preliminaries from lattice geometry.}

In order to prove Theorem \ref{thh1}.2, we aim at constructing enough loops $\alpha$ to generate $H^{\oplus d}/L^{\oplus d}$ with the respective elements $\tilde\alpha_\sigma$. For this purpose, the following obvious combinatorial fact will be useful. 

\begin{defin} For an element $u=(u_1,\ldots,u_d)\in H^{\oplus d}$, the sum $u_1+\ldots+u_d\in H$ is denoted by $\sum u$. The element $u$ is said to be \emph{homogeneous} if all of its non-zero entries are equal to each other.
\end{defin}

\begin{lemma}\label{lhomog} I. Let $U\subset H^{\oplus d}$ be a subset of homogeneous elements such that the vectors $\sum u$ over all $u\in U$ do not generate $H/L$. Then, the set $\sym_d\cdot U$ (in the sense of the natural action of the permutation group $\sym_d$ on the $d$ direct summands of $H^{\oplus d}$) does not generate the space $H^{\oplus d}/L^{\oplus d}$.

II. Let $U\subset H^{\oplus d}$ be a subset of homogeneous elements satisfying the following:

1) for any $u:=(u_1,\ldots,u_d)\in U$ with a non-zero entry $u_i$, there exists $\tilde u:=(\tilde u_1,\ldots,\tilde u_d)\in U$ 

such that $\tilde u_j=u_i$ and $\tilde u_k=0$ for some indices $j$ and $k$, and

2) the vectors $\sum u$ over all $u\in U$ generate $H/L$.\\
Then, the set $\sym_d\cdot U$ generates the space $H^{\oplus d}/L^{\oplus d}$. 
\end{lemma}
\begin{proof} In the setting of part I, the map $\sum$ sends $\sym_d\cdot U$ and $L^{\oplus d}$ to the proper sublattice of $H$ generated by $\sum  u,\, u\in U$, and $L$, thus $\sym_d\cdot U$ and $L^{\oplus d}$ also generate a proper sublattice of $H^{\oplus d}$.

In the setting of Part II, the assumption 1) allows to obtain, starting from a homogeneous element $u\in U$ with the nonzero entry $\delta$, the element in $\sym_d\cdot U$ of the form $(\delta,\ldots,\delta,0,\ldots,0)$ with at least one zero, then the element $(\delta,\ldots,\delta,0,\delta,0\ldots,0)$ with the same number of zeroes, then, by permuting the difference of the preceding two vectors, the element $(0,\ldots,0,\delta,0,\ldots,0,-\delta,0,\ldots,0)$ with $\delta$ and $-\delta$ at arbitrary positions, and finally, adding such elements to the initial $u$, the vector $(0,\ldots,0,\sum u,0,\ldots,0)$ with $\sum u$ at an arbitrary position. By the assumption 2), such vectors together with $L^{\oplus d}$ generate $H^{\oplus d}$.
\end{proof}
Note that Part II does not hold without the assumption 1. Eventually, for this reason, we shall need the following elementary geometric fact.
\begin{lemma}\label{lexcl} 
Under the assumptions of Theorem \ref{thh1}.2, assume that for some $B\in\mathcal{E}_0$ and the corresponding $\gamma:=\gamma_B$, every $j=1,\ldots,n$, and every point $a\in A_j\setminus A_j^\gamma$, we have $V(A)=h_a\cdot d_\gamma$ where $h_a:=|\gamma(a)-\gamma(A_j^\gamma)|$ is the lattice distance from $a$ to the hyperplane $A_j^\gamma+\ker\gamma$. Then 
the lattice mixed volume of $A_1,\ldots,A_n$ equals 1. In particular, by \cite{mv1}, all $A_i$ are equal to subsets of the set of vertices of the same elementary lattice simplex up to a shift.
\end{lemma}
\begin{proof} By monotonicity of the mixed volume, we have $V(A)\geqslant V \big(a\cup B_j, \left\lbrace B_i \right\rbrace_{i\neq j}\big)= h_a V_{j}$, where $V_{j}$ is the $(n-1)$-dimensional lattice mixed volume of the convex hulls of $B_i=A_i^\gamma,\, i\ne j$. Since all of these convex hulls are $(n-1)$-dimensional (by definition of $\mathcal{E}_0$), the mixed volume $V_j$ is a positive multiple of $d_\gamma$. Thus, for every $j$ and $a$, we have $V_j=d_\gamma$, and $h_a=h$ does not depend on $a$ (and $j$). In particular, in the minimal lattice $L'$ containing every $B_i$ up to a shift, there exists an elementary simplex $S$ containing every $B_i$ up to a shift. Thus, up to a shift, every $A_i$ consists of a subset of the set of vertices of $S$ together with some (possibly none) points at  lattice distance $h$ from $L'$ (on the same side from it). In this case, the equality $V(A)=h_a\cdot d_\gamma=h\cdot\Vol S$ and the monotonicity of the mixed volume imply the existence of a point $a_0$ such that every $A_i$ consists of a subset of the union of the set  vertices of $S$ with $a_0$. This implies that $\Vol S=h=1$, otherwise $A$ would not not be reduced.
\end{proof}
\smallskip
\subsection{Preliminaries from toric geometry.} 
In order to prove Theorem \ref{thh1}.2, we shall need to construct a certain deformation of the system of equations $f_\circ=0$, which is (slightly) degenerate in the sense of the theory of Newton polyhedra.

In order to prove the existence of a sought deformation, we briefly recall how to construct toric compactifications of non-degenerate complete intersections, and how they degenerate.
\begin{defin}
For a system of polynomial equations $f_1=\ldots=f_k=0$, we denote the set of its solutions by $Z(f)$, say that it is a \emph{complete intersection} if its codimension equals $k$, and say that it is \emph{regular} if 0 is a regular value of the polynomial map $f=(f_1,\ldots,f_k)$.
\end{defin}
\begin{rem}
1) Adopting a standard harmless abuse of terminology, we characterize the aforementioned property of the tuple $f$ (not of $Z(f)$ itself) by saying that $Z(f)$ is regular. 

2) In what follows, $f_i$ are Laurent polynomials, and $Z(f)$ belongs to the complex torus $T$.
\end{rem}

We consider the character lattice $\Z^n$ of the complex torus $T=\CC^n$ and its dual lattice $H=H_1(T)$. Given a finite set $B\subset \Z^n$, a polynomial $g\in\C^B$, and a vector $\gamma\in H$, we define the \emph{$\gamma$-leading term} of $g=\sum_{b\in B}c_bx^b$ as $g=\sum_{b\in B^\gamma}c_bx^b$; for generic $g\in\C^B$, this is the highest non-zero homogeneous component of $g$ with respect to the degree $\deg x^b:=\gamma(b)$.

Similarly, given a tuple of finite sets $A=(A_1,\ldots,A_m)$ in $\Z^n$ and a tuple of polynomials $f=(f_1,\ldots,f_m)\in\C^A$, we define $f^\gamma:=(f^\gamma_1,\ldots,f^\gamma_m)$.
For a generic tuple $f\in\C^A$, the system of equations $f^\gamma=0$ defines the regular zero locus $Z(f^\gamma)$.

Let us briefly recall how the sets $Z(f^\gamma)$ (over all $\gamma$) can be used to construct a smooth compactification of $Z(f)$. We  refer to \cite{KKMS} and \cite{khovcomp} for terminology and facts on toric compactifications. 

If a collection of vectors $V\subset H$ can be completed to a lattice basis, then the set of all strictly positive integer combinations of these vectors is called \emph{the simple cone} generated by $V$, and its \emph{faces} are defined as the cones generated by the subcollections of $V$. A cone $C\subset H$ is said to be \emph{compatible} with a set $A\subset\Z^n$, if the support set $A^\gamma$ does not depend on the choice of $\gamma\in C$.
\emph{A simple fan} $\Sigma$ in $H$ is a finite collection of non-intersecting simple cones covering $H$ and closed with respect to taking faces. 

\begin{rem}
Note that, for the sake of brevity, by a simple cone we mean the set of lattice points in a relatively open simple cone, and by a simple fan we mean a complete simple fan.
\end{rem}

Every simple fan $\Sigma$ gives rise to a smooth $n$-dimensional \emph{toric variety} $X_\Sigma$ with an action of the torus $T$. Orbits of this action are in correspondence with cones of $\Sigma$. This correspondence reverses the dimensions and adjacencies and sends the 0-dimensional cone $\{0\}\in\Sigma$ to the dense orbit identified with the torus $T$, so that $T$ can be regarded as a subset of $X_\Sigma$.

\begin{theor}[\cite{KKMS}]\label{thkkms}
For every tuple of finite sets $A=(A_1,\ldots,A_m)$ in $\Z^n$, there exists a simple fan $\Sigma$ compatible with each of them (i.e. such that every cone in $\Sigma$ is compatible with each of the sets $A_i$). Moreover, given a simple cone $C$ compatible with every $A_i$, the fan $\Sigma$ can be chosen to contain $C$. 
\end{theor}
If the simple fan $\Sigma$ is compatible with $A$, then, for a generic tuple of polynomials $f=(f_1,\ldots,f_m)\in\C^A$, the closure of the algebraic set $Z(f)\subset T$ in the toric compactification $X_\Sigma$ is a smooth compact algebraic variety $\overline{Z(f)}$. This smooth compactification splits into its intersections with the toric orbits, and each of these intersections can be explicitly described by polynomial equations. 

Namely, let $O$ be the orbit of $X_\Sigma$, corresponding to a cone $C\in\Sigma$. This orbit is the quotient of the torus $T$ by the subtorus $T_C$, whose homology is the vector span of $C\subset H$. Since the fan is compatible, the tuple $f^\gamma$ does not depend on the choice of $\gamma\in C$, so we denote it by $f^C$.
\begin{theor}[\cite{khovcomp}]\label{thknovkomp} Let $\Sigma$ be a simple fan compatible with a tuple $A$, $f\in \C^A$ be a generic system of equations, and $C$ be any cone in $\Sigma$.
Then the zero locus $Z(f^C)$ is regular and the intersection of the closure $\overline{Z(f)}\subset X_\Sigma$ with the corresponding orbit $O\subset X_\Sigma$ is transversal.
In particular, $\overline{Z(f)}$ is smooth and splits into smooth strata $Z(f^C)/T_C$ over all cones $C\in\Sigma$.
\end{theor}

In what follows, we shall have to deal with generic systems of equations that do not satisfy the assumption of this theorem. The set of all such systems of equations is called the \emph{$A$-bifurcation set}. 
\begin{rem}
In other words, $f$ is in the $A$-bifurcation set if $Z(f^\gamma)$ is not regular for some $\gamma$. In particular, the definition of the $A$-bifurcation set does not depend on the choice of the fan $\Sigma$ compatible with $A$.
\end{rem}
 
For $m=n$, the $A$-bifurcation set is the union of the resultant sets $R_B^{red}\subset\C^A$ over all essential tuples $B\in\mathcal{E}$, as described in Section \ref{sec:res}, and the $A$-discriminant (the closure of all $f\in\C^A$ such that $Z(f)$ is not regular). We now formulate a similar description for arbitrary $m<n$, given in \cite{adv} (in what follows, we shall particularly need the case $n=m+1$).

\begin{defin}
For a tuple $A$ of finite sets $A_1,\ldots,A_m\in\Z^n$, its \emph{Cayley configuration} $\mathcal{A}\subset\Z^n\times\Z^m$ is the set $\cup_{i=1}^m A_i\times\{e_i\},$ where $e_1,\ldots,e_m$ is the standard basis in $\Z_m$.
\end{defin}
The spaces $\C^A$ and $\C^\mathcal{A}$ can be identified by sending a tuple $f=(f_1,\ldots,f_m)\in\C^A$ to the Laurent polynomial $F=\sum_{i=1}^m\lambda_i f_i(x)$ of the variables $\lambda$ and $x$.
\begin{defin}
1) The Cayley discriminant $D^C_A\subset\C^A\simeq\C^\mathcal{A}$ is the \emph{$\mathcal{A}$-discriminant}, i.e. the closure of all $F\in\C^\mathcal{A}$ for which 0 is a critical value.

2) The tuple $A$ is said to be \emph{Cayley-effective}, if its Cayley discriminant has codimension 1, and \emph{Cayley-defective} otherwise.
\end{defin}
\begin{rem}\label{discrres}
1) The latter condition makes sense, because the $\mathcal{A}$-discriminant is a proper irreducible algebraic set in $\C^\mathcal{A}$ (see \cite{GKZ}).

2) Assume that the dimension of the affine span of $A_1+\ldots+A_m$ is smaller than $m$, then the results of Section 1 in \cite{S94} can be rephrased as follows. The tuple $A$ is Cayley-effective if and only if no $k$ sets in the tuple $A$ can be shifted to the same $(k-1)$-dimensional space for $0<k<m$. Under this assumption, the dimension of the affine span of $A_1+\ldots+A_m$ equals $m-1$, and the Cayley discriminant $D^C_A$ is the $A$-resultant $R^{red}_A$. 
\end{rem}
A tuple $B$ is said to be a \emph{facing} of a tuple $A=(A_1,\ldots,A_m)$ of finite sets in $\Z^n$, if there exists a vector $\gamma$ such that $B_i$ is either empty or equal to $A_i^\gamma$ for every $i=1,\ldots,m$. For every facing $B$, the forgetful projection $\C^A\to\C^B$ is defined by sending the tuple $f$ of polynomials $\sum_{a\in A_i}c_{a,i}x^a$ to the tuple $f|_B$ of polynomials $\sum_{a\in B_i}c_{a,i}x^a$. The preimage of the Cayley discriminant $D^C_B$ under this projection will be called and denoted in the same way.
\begin{theor}\label{bifdecomp}
1) For all effective facings $B$ of a tuple $A$, the irreducible hypersurfaces $D^C_B\subset\C^A$ are pairwise distinct.

2) Every codimension 1 irreducible component of the regular $A$-discriminant coincides with the Cayley discriminant $D^C_B$ for some effective subtuple $B$.

3) Every codimension 1 irreducible component of the $A$-bifurcation set coincides with the Cayley discriminant $D^C_B$ for some effective facing $B$.
\end{theor}
\begin{proof}
Part 1 follows from the fact (noticed in \cite{dcg}) that the defining equation of the hypersurface $D^C_B$ non-trivially depends on every coordinate in the space $\C^B$. Part 2 follows from Theorem 2.31 in \cite{dcg} (which moreover explicitly describes all such subtuples). Part 3 follows from Proposition 1.11/4.10 in the arxiv/journal version of \cite{adv} (which moreover explicitly describes all such facings).
\end{proof}

\begin{sledst}\label{corolgener}
Let $A:=(A_1,\cdots,A_m)$ be a tuple of finite sets in $\Z^n$ and $A^\gamma$ be an effective facing. Then, for a generic system of equations $f$ in the Cayley discriminant $D^C_{A^\gamma}\subset\C^A$ and any other facing $B$, the zero locus $Z(f|_B)$ is regular. In particular, in a toric compactifictaion $X_\Sigma\supset T$ whose fan $\Sigma$ is compatible with $A$, the closure of the zero locus $Z(f)$ is smooth outside the orbit whose cone contains $\gamma$, and the closure of the zero locus $Z(\hat{f})$ for any proper subtuple $\hat{f}$ of the tuple $f$ is smooth everywhere.
\end{sledst}
\begin{proof}
The regularity of $Z(f|_B)$ follows from Theorem \ref{bifdecomp}, the smoothness of the closure follows from Theorem \ref{thknovkomp}.
\end{proof}

\smallskip
\subsection{Proof of Theorem \ref{thh1}.2.}
From the start, we assume that $A$ does not satisfy the assumption of Lemma \ref{lexcl}, otherwise $d=1$ by this lemma and \eqref{eq:inductiveconn} is satisfied by assumption. The statement follows then from Theorem \ref{mainth}. 

According to Lemma \ref{lhomog}, it is enough to construct loops $\alpha$ in $\C^A\setminus\{$bifurcation set$\}$ such that the corresponding elements $\tilde\alpha_\sigma\in H^{\oplus d}$ are homogeneous, and the vectors $\sum\tilde\alpha_\sigma$ generate $H/L$. In this section, we shall construct such loops explicitly, starting from the primitive covector $\gamma=\gamma_B$, corresponding to an arbitrary essential tuple $B=A^\gamma\in\mathcal{E}_0$, an arbitrary integer $j\in\{1,\ldots,n\}$ and an arbitrary point $a\in A_j\setminus B_j$.

With no loss of generality, we assume that $0\in B_i$ for every $i=1,\ldots,n$ (otherwise we can shift $A_i$ accordingly). We should now make some consecutive choices to construct the sought loop. 
For a given $j \in \{1,\ldots, n\}$, a given $a \in A_j \setminus B_j$ and given $g=(g_1,\ldots,g_n)\in\C^B$ and $\tilde g=(\tilde g_1,\ldots,\tilde g_n)\in\C^A$, define the tuple 
$$ f_{j,t,g}(x):=x^a+g_{j}(x)+t\cdot \tilde g_j(x) \; \;  \mbox{ and } \; \;  f_{i,t,g}(x):=g_{i}(x)+t\cdot \tilde g_i(x)$$
for any $i\in \{1,\ldots, n\}\setminus \{j\}$.
We will occasionally denote $f_{j,t,g}(x)=:F_{j,g}(x,t)=:F_{j}(x,t)$ when the dependence in the parameters $t$ and $g$ is clear from the context.

From now on and until the end of this section, we choose generic tuples $g=(g_1,\ldots,g_n)$ in $R^{red}_B$ and $\tilde g=(\tilde g_1,\ldots,\tilde g_n)$ in $\C^A$ (which means that subsequently we shall use certain properties of $g$ and $\tilde g$ that are satisfied for all pairs $(g,\tilde g)\in R^{red}_B\times\C^A$ outside a certain proper Zariski closed subset), and
call $f_{1,t_0,g}=\ldots=f_{n,t_0,g}=0$ for small $t_0\ne 0$ the \emph{central system of equations}. 
The name is explained by the following key lemma that, together with Lemma \ref{lhomog}, proves Theorem \ref{thh1}.2.
\begin{lemma}\label{lcentralloop}
Let $G:(\C,0)\to(\C^B,g)$ be a germ of an analytic curve transversal to the resultant $R^{red}_B$ at its smooth point $g$ and choose sufficiently small $|t_0|\ll\varepsilon\ll 1$. As the argument of $G$ travels $-\gamma(a)$ times along the loop $\varepsilon\exp(2\pi i s), s\in[0,1]$, the roots of the corresponding system of equations  
\begin{equation}\label{eq:loop}
    f_{\bullet,\, t_0,\, G(\varepsilon\exp(2\pi i s))}=0
\end{equation}
 permute trivially.
Moreover, $-d_\gamma\cdot\gamma(a)$ of the roots of this system travel a loop in $T$ whose homology class equals $\gamma\in H$ and the other roots travel a contractible loop.
\end{lemma}
\begin{rem}
The roots of the family of systems of equations (\ref{eq:loop}) split into ``travelling'' roots and ``still'' ones. If $t_0$ vanishes, the ``still'' roots hide at infinity, but ``travelling'' roots of the resulting family $ f_{\bullet,\, 0,\, G(\varepsilon\exp(2\pi i s))}=0$ keep traveling as described in the lemma. So the only purpose of including monomials with $t$ in the definition of $f_{j,t,g}$ (and setting this $t$ to be non-zero albeit small) is to drive the ``still'' roots back from infinity, while not affecting the itinerary of the travelling roots.
\end{rem}
The rest of this subsection is devoted to the proof of this lemma. This is actually the multidimensional version of step $(III)$ in the proof of Lemma \ref{ldim1nonred}, but the statement and the proof are significantly more technical.

\vspace{1ex}

{\it I: describing the roots of the central system of equations as $t_0 \to 0$.} The polynomials $F_j:=F_{j,g}$ are defined on the complex torus $T'=\CC^n\times\CC^1$ with the standard coordinates $(x,t)$, and their respective Newton polytopes $\Delta_j$ are contained in the character lattice $\Z^n\times\Z^1$ of $T'$, see Figure \ref{fig:compact} for an example.

Note that every strictly positive linear combination $\delta$ of the covectors $(0,-1)$ and $(\gamma,0)\in(\Z^n\times\Z^1)^*$ supports the same face $\Delta_i^\delta$: this face is the convex hull of $B_i\times\{0\}\subset\Z^n\times\Z^1$. Thus, by Theorem \ref{thkkms}, the simple cone $C$ generated by $(0,-1)$ and $(\gamma,0)$ can be extended to a simple fan $\Sigma$ compatible with the polytopes $\Delta_i$.
Let $X$ be the corresponding smooth toric compactification of $T'$. Denote its orbit corresponding to $C$ by $O$, 
and the adjacent codimension 1 orbits corresponding to the vectors $(0,-1)$ and $(\gamma,0)$ by $O_{(0,-1)}$ and $O_{(\gamma,0)}$ respectively, see Figure \ref{fig:compact} for an example.

Denoting by $\tilde O$ the union of $O$, $T'$ and the two aforementioned codimension 1 orbits, we notice that the polynomials $F_1,\ldots,F_n$ and $t$ define regular functions on the open set $\tilde O$, because $0\in B_i$. We denote these regular functions on $\tilde O$ by the same letters.

\begin{lemma}\label{ltransv}
Let $Z_i$ be the closure of the zero locus $Z(F_i)$ in the toric variety $X$. Then, for generic choise of $(g,\tilde g)\in R^{red}_B\times\C^A$, the following transversality conditions take place.

1) The intersection of $Z_1,\ldots,Z_n$ and the orbit $O$ consists of $d_\gamma$ points (see Subsection 4.1 for the resultant multiplicty $d_\gamma$). Near each of these points $x_0$, the hypersurfaces $Z_1,\ldots,Z_n$ are smooth and intersect each other transversally along a smooth curve. This curve is transversal to the closure of $O_{(\gamma,0)}$ and tangent of order $-\gamma(a)$ to the closure of the orbit $O_{(0,-1)}$.

2) If a point $x_0\in X$ belongs to the hypersurfaces $Z_{i_1},\ldots,Z_{i_k}$ and the closures of the codimension 1 orbits $C_1,\ldots,C_m$, then these $k+m$ hypersurfaces are smooth and mutually transversal at $x_0$, except for the setting of part (1), i.e. unless $x_0\in O$, the closure of $O_{(0,-1)}$ is among $C_1,\ldots,C_m$, and $k=n$.
\end{lemma}
\begin{rem}\label{remtransv}
As we shall see from the proof,  part 1 is valid for a generic choice of $g\in R^{red}_B$ and $\tilde g_{j'}\in \C^{B_{j'}}$ for some $j'$, even if $\tilde g_i$ are chosen to be 0 for $i\ne j'$.
\end{rem}
\begin{proof}
{\it Part 2} is a special case of Corollary \ref{corolgener}, because the resultant $R^{red}_B$ coincides with the Cayley discriminant $D^C_B$ by Remark \ref{discrres}.2.

{\it Transverality of $Z_1,\ldots,Z_n$ and $O_{(\gamma,0)}$.} By Theorem \ref{thm:degres}.2, the differentials of $g_i$, $i\in \{1,\cdots,n\}$, at $x_0$ satisfy a unique linear relation, thus the differentials of $g_i+t\tilde g_i^\gamma$ are linearly independent for generic choice of $\tilde g_i$. Since $g_i+t\tilde g_i^\gamma$ is the restriction of $F_i$ to the orbit $O_{(\gamma,0)}$, this implies transverality of $Z_1,\ldots,Z_n$ and the closure of $O_{(\gamma,0)}$.

{\it Order of tangency of the sought curve with $O_{(0,-1)}$.} This order equals the order of tangency of the hypersurface $Z_j=\{F_j=0\}$ with the the intersection of the hypersurfaces $Z_i,\, i\ne j$, and the closure of $O_{(0,-1)}$. Part 2 ensures that the latter intersection is transversal and defines a smooth curve. By the definition of its defining equations, this smooth curve is a shifted one-dimensional torus $T_\gamma\subset O_{(0,-1)}\simeq T$, and the restriction of $F_j$ to it equals $x^a$, thus the sought order of tangency equals $-\gamma(a)$.
\end{proof}

Recall that we are interested in the roots of the central system of equations. Identifying the subtorus $t=t_0$ in $T'$ with $T$, the roots of the central system are the intersection points of the curve $F_1=\ldots=F_n=0$ and the subtorus $t=t_0$. 
Consider the closure $C$ of this curve in the toric variety $X$. Lemma \ref{ltransv} ensures that $C$ is smooth.

The coordinate function $t$ on the torus $T'$ extends to a meromorphic function on $X$, and the equation $t=0$ defines a normal crossing divisor supported at the closure of certain codimension 1 orbits $O_m\subset X,\, m=1,\ldots,M$.
By Lemma \ref{ltransv}, the restriction of $t$ to the smooth curve $C$ has the following (and no other) roots: 

-- finitely many (possibly multiple) roots $x_{m,k}$ in the codimension 1 orbits $O_m$;

-- $d_\gamma$ roots $x_k$ of multiplicity $-\gamma(a)$ in the codimension 2 orbit $O$. 

The roots of the central system of equations tend to $x_{m,k}$ and $x_k$ as $t_0$ tends to $0$, so we need a convenient coordinate system around every $x_{m,k}$ and $x_k$ to analyze the roots of the central system for small $t_0$.

\vspace{1ex}

{\it II: describing a good neighborhood of a root $x_{m,k}\in O_m$.}  
By Lemma \ref{ltransv}, every root $x_{m,k}$ admits an open neighborhood $U_{m,k}\subset X$ with an analytic coordinate system $(y_1,\dots,y_n,\tau)$ such that:

-- the coordinate $y_i$ is given by $F_i$ for any $i\in\{1,\dots,n\}$;

-- the coordinate $t$ is given by $\tau^\mu$ for some positive integer $\mu$;

-- $U_{m,k}$ does not intersect orbits of $X$ outside $T'\cup O_m$;

-- the intersection of $U_{m,k}$ with the plane $\tau=\tau_0$ for every sufficiently small $\tau_0$ is a topological disc.

This implies the existence of neighbourhoods $V_{m,k}\subset \C^B$ of $g$ and $W_{m,k}\subset \C$ of $0$ such that, for every $\tau_0,\,\tau_0^\mu\in W_{m,k}$, and every $g'\in V_{m,k}$, the curve $F_{1,g'}=\cdots=F_{n,g'}=0$ transversally intersects the disc $U_{m,k}\cap\{\tau=\tau_0\}$ at one point.

\newpage

\begin{figure}[h]
\begin{center}
\includegraphics[width=14cm]{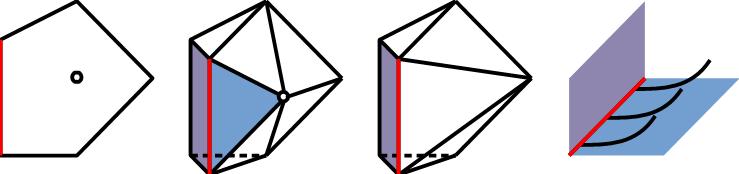}
\end{center}
\captionsetup{width=.9\linewidth}
\caption{
\textbf{An example of the toric compactification $\mathbf{X}$ in the case $\mathbf{n=2}$}.
On the left, we represent the convex hull of $A_1=A_2$ with the red edge $B:=B_1=B_2$ and the point $a$. The two figures in the middle represent the polytopes $\Delta_i$ and $\Delta_j,\, j\ne i$, with the red edge $B\times\{0\}$: the blue facet supports the covector $(0,-1)$ and the purple ones the covector $(\gamma,0)$. On the right, we picture the (real part of the) orbits $O$, $O_{(0,-1)}$ and $O_{(\gamma,0)}$ in $X$ corresponding to the faces of the polytopes $\Delta_j$ of the respective colors, and the black curve $F_1=F_2=0$ near $O$.
}
\label{fig:compact}
\end{figure}

\begin{lemma}\label{lxmk}
Assume that, in the setting of Lemma \ref{lcentralloop}, $\varepsilon$ is so small that the loop $G(\varepsilon\exp(2\pi i s)),\, s\in[0,1]$, is contained in $V_{m,k}$, and $t_0\in W_{m,k}$. As the system $f_{\bullet,\, t_0,\, G(\varepsilon\exp(2\pi i s))}=0$ travels along this loop, its roots in $\{t=t_0\}\cap U_{m,k}$ permute trivially, and each of them travels a contractible loop in the torus $\{t=t_0\}\simeq T$.
\end{lemma}
\begin{proof}
By the choice of the neighborhoods $U_{m,k}$, $V_{m,k}$, $W_{m,k}$, the intersection $\{t=t_0\}\cap U_{m,k}$ consists of $\mu$ disjoint discs in $T$, and exactly one of the roots varies in each of these discs. Thus the roots do not permute and travel contractible loops.
\end{proof}

\vspace{1ex}

{\it III: describing a good neighborhood of a root $x_{k}\in O$.}

for every root $x_{k}$ and every choice of $j'\in\{1,\ldots,n\}$,
there exists an open neighborhood $U_{k}\ni x_k$ with an analytic coordinate system such that:

-- $t$ and $F_i,\, i\ne j'$, are $n$ of the $n+1$ coordinate functions;

-- if $\ell$ is the coordinate line defined by them, and $\varphi$ is the remaining coordinate function, then the restriction of $F_{j'}$ to $\ell$ equals $\varphi^{-\gamma(a)}$;

-- $U_{k}$ does not intersect orbits of $X$ outside $\tilde O$;

-- the intersection of $U_{k}$ with the plane $t=t_0$ for every sufficiently small $t_0$ is a topological disc.

This implies the existence of neighbourhoods $V_{k}\subset \C^B$ of $g$ and $W_{k}\subset \C$ of $0$ such that, for every non-zero $t_0\in W_k$ and $g'\in V_k$, the curve $F_{1,g'}=\cdots=F_{n,g'}=0$ transversally intersects the disc $U_k\cap\{t=t_0\}$ at $-\gamma(a)$ points.

\begin{lemma}\label{lxk}
Assume that, in the setting of Lemma \ref{lcentralloop}, $\varepsilon$ is so small that the loop $G(\varepsilon\exp(2\pi i s)),\, s\in[0,1]$, is contained in $V_{k}$, and $|t_0|\ll\varepsilon$. As the system $f_{\bullet,\, t_0,\, G(\varepsilon\exp(2\pi i s))}=0$ travels along this loop, all of its its $-\gamma(a)$ roots in $U_{k}\cap\{t=t_0\}$ permute cyclically, and their paths form a loop in $\{t=t_0\}\simeq T$, whose homology class equals $\gamma\in H$.
\end{lemma}
\begin{proof}
First, we assume with no loss of generality that $t_0=0$, because the statement is invariant under perturbations of $t_0$. Second, we assume without loss of generality that $G$ is given by $G_{j'}(\epsilon)=g_{j'}+\epsilon\tilde g_{j'}$ for generic $g_{j'}\in\C^{B_{j'}}$ and $G_i(\epsilon)=g_i$ for $i\ne j'$, because all small loops around the resultant $R_B^{reg}$ near its smooth point $g$ are homotopy equivalent.

Under these additional assumptions, the sought statement turns into the following question: given the curve $C'\subset T'$ defined by the system of equations $f_{\bullet,0,G(t)}(x)=0$, and the point $x_k$ in its closure, what happens to the points of the intersection $C'\cap\{t=t_0\}$ near $x_k$ as $t_0$ runs once around $0$?

The answer is the expected one: the paths of the intersection points form a loop in $\{t=t_0\}\simeq T$, whose homology class equals $\gamma\in H$. This is because, by Lemma \ref{ltransv}.1 (which is applicable to the equations $f_{\bullet,0,G(t)}(x)=0$ according to Remark \ref{remtransv}), the point $x_k$ is a smooth point of the closure of $C'$, and the restriction of $t$ to this closure has a root of multiplicity $-\gamma(a)$ at $x_k$.
\end{proof}

\vspace{1ex}

{\it Proof of Lemma \ref{lcentralloop}.} For $\varepsilon$ and $|t_0|$ small enough, all roots of the  system of equations ${(*)}$ belong to the neighborhoods $U_{m.k}\cap\{t=t_0\}$ and $U_{k}\cap\{t=t_0\}$, because all intersections of the hyperplane $t=t_0$ with the curve 
$F_{1,g'}=\cdots=F_{n,g'}=0$? tend to the points $x_{m,k}$ and $x_k$ as $t_0\to 0$ and $g'\to g$. In these neighborhoods, for $|t_0|\ll\varepsilon$, the roots permute as desired according to Lemmas \ref{lxmk} and \ref{lxk}. \hfill$\square$

\vspace{1ex}

{\it Proof of Theorem \ref{thh1}.2.} 
Given a number $j\in\{1,\ldots,d\}$, a vector $\gamma:=\gamma_B,\, B\in\mathcal{E}_0$, and a point $a\in A_j\setminus B_j$, choose any path $\beta$ from $f_\circ$ to the base point of the loop constructed in Lemma \ref{lcentralloop}. Conjugating the latter loop with $\beta$, we obtain a loop $\alpha_{\beta,\gamma,j,a}$ based at $f_\circ$.

Choosing an ordering $\delta:\{$roots of the base system $f_\circ=0\}\to\{1,\ldots,d\}$, the loop $\alpha_{\beta,\gamma,j,a}$ gives rise to the corresponding vector $\tilde\alpha_{\delta,\beta,\gamma,j,a}$ in the solution lattice $H_\delta\subset H^{\oplus d}$ (recall that its $i$-th component equals the homology class of loop run by the $i$-th root of $f_\circ=0$ as $f_\circ$ travels along $\alpha_{\beta,\gamma,j,a}$). 

According to Lemma \ref{lcentralloop}, all non-zero components of this vector are equal to $\gamma\in H$, and the number of them equals $(\gamma(B_j)-\gamma(a))\cdot d_\gamma$. Moreover, every vector of this form equals $\tilde\alpha_{\delta,\beta,\gamma,j,a}$ for suitable $\beta$, since the monodromy group of the general reduced irreducible system of equations equals $\sym_d$ (see \cite{E18}).

We conclude that, for every $\gamma,j,a$ as above, every vector with $(\gamma(B_j)-\gamma(a))\cdot d_\gamma$ components equal to $\gamma$ and other components equal to 0 is contained in the solution lattice $H_\delta$. 

Note that, for a given $\gamma$, the numbers $\gamma(B_j)-\gamma(a)$ over all $j$ and $a$ are mutually prime (otherwise, if GCD equals $k>1$, then the sets of tuple $A$ can be shifted to the proper sublattice $\gamma^{-1}(k\Z)$, so $A$ is not reduced). Thus the Euclidean algorithm ensures that every vector with $d_\gamma$ components equal to $\gamma$ and other components equal to 0 is contained in the solution lattice $H_\delta$.

By the assumption of Theorem \ref{thh1}, the sums of components of such vectors generate $H/L$, thus by Lemma \ref{lhomog}.II the vectors themselves generate $H^{\oplus d}/L^{\oplus d}$ (the condition (1) of this lemma satisfied by Lemma \ref{lexcl}).\hfill$\quad$

\begin{rem}\label{remincr2} Actually, the construction from the preceding proof is applicable to $B\in\mathcal{E}$ even if $B\notin\mathcal{E}_0$.  Choose an arbitrary tuple $C=(C_1,\ldots,C_n),\,C_i\subset A_i$, generic $\tilde f=(\tilde f_1,\ldots,\tilde f_n)\in\C^{C}$ and a small loop $g_s\in \C^B,\, s\in S^1$, around the resultant $R^{red}_B$. For every linear function $\gamma\in\mathcal{G}_B$, let $h_{\gamma}$ be $\max_i (\max\gamma|_{A_i}-\max\gamma|_{C_i})$. Set
$f_{i,t,s}(x)=g_{i,s}(x)+\tilde f_i(x)+t\cdot \tilde g_i(x)$ and let $s\in S^1$ run a loop for a small $t\ne 0$. Then, as in the preceding proof, the roots of the system $f_{\bullet,t,s}=0$ permute so that (cf. Proposition 5.2 in \cite{mathann}):

1) among the disjoint cycles of the permutation of the roots, we have $d_\gamma$ cycles of length $h_\gamma$ for every $\gamma\in\mathcal{G}_B$;

2) the paths of the roots from one cycle constitute a loop whose class in the homology $H$ equals $\gamma$.

As $s\in S^1$ runs around the circle sufficiently many times (more specifically, $M={\rm LCM}\{h_\gamma\,|\,\gamma\in\mathcal{G}_B\}$ times), we obtain a certain element $\tilde\gamma_B\in H^{\oplus d}$ in the homology of the solution space. According to 1) and 2), this element $\tilde\gamma_B$ has $d_\gamma$ entries equal to $\frac{M}{h_\gamma}\gamma$ for every $\gamma\in\mathcal{G}_B$, and the other entries are equal to 0. As in the preceding proof, we now have four increasing classes, extending Remark \ref{remincr}.2:

a) Tuples $A$, such that the vectors  $d_{\gamma_B}\cdot{\gamma_B}$ over all $B\in\mathcal{E}_0$ generate the lattice $H$;

b) Tuples $A$, such that the vectors $\tilde\gamma_B$ over all $B\in\mathcal{E}$ generate the lattice $H^{\oplus d}$;

c) Tuples with inductively connected solution spaces;

d) Tuples $A$, such that the vectors $\sum_{\gamma\in\mathcal{G}_B} d_\gamma\cdot\gamma$ over all $B\in\mathcal{E}$ generate the lattice $H$.

It is now a purely combinatorial (although highly non-trivial) problem to understand whether the classes (b) and (d) coincide for all reduced irreducible tuples $A$. If the answer is ``yes'' (and this is what we expect at least for $n=2$), then we have (b)$=$(c)$=$(d), so Theorem \ref{thh1}.1 actually provides a criterion of whether the Galois group of a given tuple equals the expected wreath product. If the answer is ``no'', then a more subtle study of solution spaces is required to answer this question.
\end{rem}
\smallskip
\subsection{Proving inductive disconnectedness.}
In order to prove that a given tuple $A$ is not inductively connected, we need the following Poisson-type formula for the product of roots of a system of polynomial equations. It is the special case of the Poisson-Pedersen-Sturmfels-D'Andrea-Sombra formula \cite[Theorem 1.1]{ds}, when one of the $n+1$ polynomials involved is a monomial $x^b$.
\begin{theor} For a generic system of equations $ f=(f_1,\ldots,f_n)\in\C^A$, the product of the values of the monomial $x^b$ over the roots of  $f_1=\ldots=f_n=0$ equals $\prod_{\gamma\in\mathcal{G}}[R_{A^\gamma}(f)]^{\gamma\cdot b}$.
\end{theor}

\begin{proof}[Proof of Theorem \ref{thh1}.1.]
Under the assumptions of Theorem \ref{thh1}.1, we can choose $b\in \Z^n$ and $p>1$ that divides $d_\gamma\cdot(\gamma\cdot b)$ and $l\cdot b$ for all $\gamma\in\mathcal{G}$ and all $l \in L$. 
Recall that the polynomial $R_{A^\gamma}$ is equal to $F^{d_\gamma}$ for some polynomial $F$ on $\C^A$. In this case, the preceding Poisson-type formula implies that the  product of the monomial $x^b$ over the roots of $f$ equals $G^p$ for some polynomial $G$ on $\C^A$. Moreover, the polynomial $G$ does not vanish at systems of equations $f$ that have $d$ distinct roots in the torus since the irreducible factors of $G$ describe systems $f$ with roots at infinity. Thus, for any loop $\alpha$ in the space of systems with $d$ distinct roots, the composition $G^p\circ \alpha$ is a loop in $\C^\star$ whose homology class is divisible by $p$. The latter homology class is given by $(\sum\tilde\alpha_{\sigma})\cdot b$, since $G^p(\alpha(s))$ is the product of the monomials $x^b$ over the roots of the system $\alpha(s)$. We deduce that $(\sum\tilde\alpha_{\sigma})\cdot b$ is divisible by $p$. 
All such elements together with the sublattice $L^{\oplus d}$ generate a proper sublattice in $H^{\oplus d}$, because it is contained in the proper sublattice of all $u\in H^{\oplus d}$ such that $b\cdot\sum u$ is divisible by $p$. Therefore $L^{\oplus d}$ and the solution lattice $H_\delta$ do not generate $H^{\oplus d}$.

\end{proof}

\textbf{Acknowledgement.} The two authors met during the program ``Tropical Geometry, Amoebas and Polytopes'' held at the Institute Mittag-Leffler in spring 2018. The authors would like to express their gratitude to the organizers J. Draisma, A. Jensen, H. Markwig, B. Nill and to the institute for providing inspiring working conditions.

\bigskip
\bigskip
\noindent
A. Esterov\\
National Research University Higher School of Economics\\
Faculty of Mathematics NRU HSE, Usacheva str., 6, Moscow, 119048, Russia\\
\textit{Email}: aesterov@hse.ru\\

\noindent
L. Lang\\
Department of Mathematics, Stockholm University, SE - 106 91 Stockholm, Sweden.\\
\textit{Email}: lang@math.su.se


\begin{thebibliography}{xxxx}

\bibitem[A57]{abh1}
 S. S. Abhyankar, {\it Coverings of algebraic curves}, Amer. J. Math. 79 (1957), 825–856.

\bibitem[ASP92]{abh2}
S. S. Abhyankar, W. K. Seiler, and H. Popp, {\it Mathieu group coverings of
the affine line}, Duke Math. J. 68 (1992), no. 2, 301 – 311.

\bibitem[ABS19]{abs}
G. Averkov, C. Borger, I. Soprunov,
{\it Classification of triples of lattice polytopes with a given mixed volume}, arXiv:1902.00891

\bibitem[B75]{bernst} D. N. Bernstein, {\it The number of roots of a system of equations}, Functional Anal. Appl. 9 (1975) 183--185. 

\bibitem[C80]{cohen80}
S. D. Cohen, {\it The Galois group of a polynomial with two indeterminate coefficients.}, Pacific J. Math., 90 (1980) 63-76 and 97 (1981) 483–486.

\bibitem[DS13]{ds} C. D'Andrea, M. Sombra, {\it A Poisson formula for the sparse resultant},
Proceedings of the London Mathematical Society, 110 (2015) 932--964, arXiv:1310.6617 

\bibitem[B07]{wreathgalois}
B. de Smit, {\it Galois groups and wreath products}, 2007.

{\tt http://www.math.leidenuniv.nl/$\sim$desmit/notes/krans.pdf}

\bibitem[E07]{mian}
A. Esterov, {\it Determinantal Singularities and Newton Polyhedra}, Proceedings of the Steklov Institute of Mathematics, 2007, 259, 16–-34, mi.mathnet.ru/eng/tm567

\bibitem[E08]{dcg}  A. Esterov, {\it Newton polyhedra of discriminants of projections}, Discrete Comput. Geom., 44 (2010) 96--148, arXiv:0810.4996

\bibitem[E11]{adv} A. Esterov, {\it The discriminant of a system of equations}, Adv. Math. 245 (2013) 534--572, arXiv:1110.4060

\bibitem[EG12]{mv1}  A. Esterov, G. Gusev, {\it Systems of equations with a single solution}, J. of Symb. Comput., 68 (2015) 116--130, arXiv:1211.6763

\bibitem[EG14]{mathann}  A. Esterov, G. Gusev, {\it Multivariate Abel--Ruffini}, Math. Ann., 365 (2016) 1091--1110, arXiv:1405.1252

\bibitem[E18]{E18} A. Esterov, {\it Galois theory for general systems of polynomial equations}, Compositio Math. 155 (2019), 229--245, arXiv:1801.08260

\bibitem[EL20]{el2} A. Esterov, L. Lang {\it Braid monodromy of univariate fewnomials}, to appear

\bibitem[F93]{Ful} W. Fulton, {\it Introduction to toric varieties}, The William H. Roever Lectures in Geometry, Annals of Mathematics Studies, Princeton University Press, 1993 

\bibitem[GKZ94]{GKZ} I. M. Gelfand, M. M. Kapranov, A. V. Zelevinsky, {\it Discriminants, Resultants, and Multidimensional Determinants}, Springer, 1994

\bibitem[H79]{Harris}
J. Harris, {\it Galois groups of enumerative problems}, Duke Math. J. 46 (1979), 685--724
   
\bibitem[Ha02]{Hatcher}
A. Hatcher, {\it Algebraic topology}, Cambridge University Press (2002), 1--544  

\bibitem[HT17]{ht}
T. Hibi, A. Tsuchiya, {\it Classification of lattice polytopes with small volumes}, to appear in Journal of Combinatorics, arXiv:1708.00413


\bibitem[KKMS]{KKMS}
G. Kempf, F. Knudsen, D. Munford, B. Saint-Donat; Toroidal Embeddings I; Lecture Notes in Math., 339, Springer-Verlag, 1973.
   
\bibitem[K77]{khovcomp} 
A. G. Khovanskii, {\it Newton polyhedra and toroidal varieties}, Functional Analysis and Its Applications, 11 (1977) 289--296.

\bibitem[L19]{L19} 
L. Lang, {\it Monodromy of rational curves on toric surfaces}, arXiv:1902.08099, 2019

\bibitem[S77]{tri2}
J. H. Smith, {\it General trinomials having symmetric Galois group}, Proc. Amer. Math.
Soc, 63 (1977), 208-212.

\bibitem[S84]{tri3}
Sur le groupe de Galois d'un polynome dont les coefficients sont ind\'ependants, Sem. Th\'eor. Nombres Bordeaux (1984), 1-11.

\bibitem[SW13]{schsot}
F. Sottile, J. White, {\it Double transitivity of Galois Groups in Schubert Calculus of Grassmannians}, Alg. Geom. 2 (2015) 422--445, arXiv:1312.5987

\bibitem[S94]{S94}
B. Sturmfels; On the Newton polytope of the resultant; J.
Algebraic Combin. 3 (1994), no. 2, 207--236.

\bibitem[T14]{Tyom14}
I. Tyomkin, \it{An example of a reducible Severi variety},
 Proceedings of the 20th G\"okova geometry-topology conference, G\"okova, Turkey, (2014)
33--40.

\bibitem[U70]{tri1}
K. Uchida, {\it Galois group of an equation $X^n
-aX+b=0$}, Tohoku Math. J., 22 (1970),
670-678.

\bibitem[V03]{Vakil}
R. Vakil, {\it Schubert induction}, Ann. Math. 164 (2006) 489--512, arXiv:math/0302296

\end{thebibliography}
\end{document}